\documentclass[a4paper,10pt,reqno]{elsarticle}
\usepackage{amsmath}
\usepackage{cases}

\usepackage{amsfonts}
\usepackage[colorlinks,linkcolor=blue,citecolor=blue]{hyperref}
\usepackage{latexsym, amssymb, amsmath, amsthm, bbm}
\usepackage[all]{xy}
\usepackage{pgfplots}
\usepackage{enumerate}

\DeclareSymbolFont{EulerExtension}{U}{euex}{m}{n}
\DeclareMathSymbol{\euintop}{\mathop} {EulerExtension}{"52}
\DeclareMathSymbol{\euointop}{\mathop} {EulerExtension}{"48}

\allowdisplaybreaks[4]

\def \id{\operatorname{id}}
\def \Id{\operatorname{Id}}

\def \C{\mathcal{C}}

\def \e{\varepsilon}
\def \M{\mathrm{M}}

\def \Z{\mathbb{Z}}

\def \k{\Bbbk}

\def \dim{\operatorname{dim}}

\def \Hom{\operatorname{Hom}}

\def \Id{\operatorname{Id}}

\def \Rep{\operatorname{Rep}}

\def \C{\mathcal{C}}
\def \D{\Delta}

\def \e{\varepsilon}
\def \M{\mathrm{M}}

\def \C{\mathcal{C}}
\def \D{\mathcal{D}}

\def \M{\mathfrak{M}}

\def \Z{\mathcal{Z}}

\def \dast{{\ast\ast}}

\def \1{\mathbf{1}}
\def \Rep{\mathsf{Rep}}
\def \op{\mathrm{op}}
\def \cop{\mathrm{cop}}

\usepackage{mathtools}
\usepackage{stmaryrd}

\def \pams{partially admissible mapping system}
\def \pd{\boldsymbol}
\def \YD{\mathfrak{YD}}

\def \la{\langle}
\def \ra{\rangle}

\def \btl{\raisebox{0.2ex}{\,${\scriptstyle \blacktriangleleft}$}\,}
\def \btr{\raisebox{0.2ex}{\,${\scriptstyle \blacktriangleright}$}\,}

\def \rev{\mathrm{rev}}

\numberwithin{equation}{section}

\newtheorem{theorem}{Theorem}[section]
\newtheorem{lemma}[theorem]{Lemma}
\newtheorem{proposition}[theorem]{Proposition}
\newtheorem{corollary}[theorem]{Corollary}
\newtheorem{definition}[theorem]{Definition}

\newtheorem{remark}[theorem]{Remark}
\newtheorem{question}[theorem]{Question}

\newtheorem{notation}[theorem]{Notation}

\begin{document}

\title{The quantum double of Hopf algebras realized via partial dualization and the tensor category of its representations}

\author[J.-W. He]{Ji-Wei He}
\address{School of Mathematics, Hangzhou Normal University, Hangzhou 311121, China}
\ead{jwhe@hznu.edu.cn}

\author[X. Kong]{Xiaojie Kong}
\address{School of Mathematics, Hangzhou Normal University, Hangzhou 311121, China}
\ead{xjkong@stu.hznu.edu.cn}

\author[K. Li]{Kangqiao Li\corref{cor1}}
\address{School of Mathematics, Hangzhou Normal University, Hangzhou 311121, China}
\ead{kqli@hznu.edu.cn}

\cortext[cor1]{Corresponding author}


%

\date{}

\begin{abstract}

In this paper, we aim to study the (generalized) quantum double $K^{\ast\cop}\bowtie_\sigma H$ determined by a (skew) pairing between finite-dimensional Hopf algebras $K^{\ast\cop}$ and $H$.
We show that $K^{\ast\cop}\bowtie_\sigma H$ is a left partially dualized (quasi-)Hopf algebra of $K^\op\otimes H$.
Using this formulation, we establish the \textit{tensor equivalences} between the following finite tensor categories
$$\Rep(K^{\ast\cop}\bowtie_\sigma H),\;\;\;\;{}_H\mathfrak{YD}^K,\;\;\;\;
{}^K_K\M^ K_H,\;\;\;\;\text{and}
\;\;\;\;{}^{K^\ast}_{K^\ast}\M^{ H^\ast}_{K^\ast},$$
that is, the categories of (finite-dimensional) representations of $K^{\ast\cop}\bowtie_\sigma H$, relative Yetter-Drinfeld modules over $(H,K)$, and two-sided two-cosided relative Hopf modules.

~

\end{abstract}

\maketitle


\section{Introduction}

The Drinfeld double $D(H)$ of a finite-dimensional Hopf algebra $H$ is an important construction due to Drinfeld \cite{Dri86},
and its theories have been widely developed, as there is a categorical observation in \cite{Kas95} that the category
$\Rep(D(H))$ of finite-dimensional representations of $D(H)$ is braided tensor equivalent to the center of $\Rep(H)$:
\begin{equation}\label{eqn: 1.1}
\Rep(D(H))\approx\Z(\Rep(H)).
\end{equation}

In 1994, Doi and Takeuchi \cite{DT94} constructed a kind of Hopf algebra determined by a skew Hopf pairing, whose properties are studied in \cite{AFG01,LMS06,RS08,Rad12,HS20} etc.
This construction is usually referred to as the (generalized) quantum double, and it is frequently regarded as a generalization of the Drinfeld double.
Since we only study finite-dimensional cases in this paper, an equivalent formulation can be used as follows:
Let $H$ and $K$ be finite-dimensional Hopf algebras over $\k$ with Hopf pairing $\sigma: K^\ast\otimes H\rightarrow \k$ (inducing Hopf algebra maps $\sigma_l$ and $\sigma_r$). Then it determines the quantum double denoted by $K^{\ast\cop}\bowtie_\sigma H$, which turns to be $D(H)$ if $K=H$ and $\sigma$ is the evaluation.

However, in order to generalize (\ref{eqn: 1.1}) for the case of quantum doubles, or to establish other tensor equivalences from $\Rep(K^{\ast\cop}\bowtie_\sigma H)$, we try to apply the notion of the left partially dualized quasi-Hopf algebra (or left partial dual for short) introduced by Li \cite{Li23}.
This is because a left partial dual of $H$ is categorically Morita equivalent to $H$, meaning that it reconstructs a certain dual tensor category of $\Rep(H)$.
Our first main result is the following one, which is a combination of Theorem \ref{prop: lpd qd} and Proposition \ref{prop: te lDoi} (or Corollary \ref{cor: te lDoi}).

\begin{theorem}\label{thm: 1.1}
Let $H$ and $K$ be finite-dimensional Hopf algebras over $\k$ with Hopf pairing $\sigma: K^\ast\otimes H\rightarrow \k$.
Then
\begin{itemize}
\item[\rm{(1)}]
The quantum double $K^{\ast\cop}\bowtie_\sigma H$ is a left partial dual of the tensor product Hopf algebra $K^\op\otimes H$, and consequently,
\item[\rm{(2)}]
$\Rep(K^{\ast\cop}\bowtie_\sigma H)$ is tensor equivalent to the category ${}^{}_{K^{\ast\cop}}\M^{K^{\ast\cop}\otimes H^\ast}_{K^{\ast\cop}}$ of relative Doi-Hopf modules, where $K^{\ast\cop}$ is regarded as a right $K^{\ast\cop}\otimes H^\ast$-comodule algebra via coaction
$k^\ast \mapsto\sum k^\ast_{(2)}\otimes\big(k^\ast_{(1)}\otimes\sigma_l(k^\ast_{(3)})\big)$.
\end{itemize}
\end{theorem}

In fact, the tensor equivalence stated in Theorem \ref{thm: 1.1}(2) above is considered to be the reconstruction of the left partial dual $K^{\ast\cop}\bowtie_\sigma H$, following \cite{Li23}. As applications, another goal of this paper is to obtain further tensor equivalences for $\Rep(K^{\ast\cop}\bowtie_\sigma H)$, which can also be analogues of the equivalences for $\Rep(D(H))$.

Let us recall some identifications of the (braided) tensor category $\Rep(D(H))$ in the literature. Majid \cite{Maj91}
showed that $\Rep(D(H))$ is isomorphic to ${}_H\YD^H$, which is known as the category of finite-dimensional (left-right) Yetter-Drinfeld modules over $H$ according to Radford and Towber \cite{RT93}.
Later, Schauenburg \cite{Sch94} provided tensor equivalences from ${}^H\YD_H$ (or equivalently, ${}_H\YD^H$ and ${}_H^H\YD$) to the category ${}_H^H\M{}_H^H$ of two-sided two-cosided Hopf modules over $H$, and this result actually holds in a symmetric monoidal category with equalizers. Moreover in 2002, he proved in \cite{Sch02} a generalization to the case when $H$ is a quasi-Hopf algebra, by extending a structure theorem of Hausser and Nill \cite{HN99}.

Now we introduce our conclusion (=Theorem \ref{them: cc}) on the quantum double which is analogous to the results mentioned above.
\begin{theorem}\label{thm: 1.2}
Let $H$ and $K$ be finite-dimensional Hopf algebras over $\k$ with Hopf pairing $\sigma: K^\ast\otimes H\rightarrow \k$.
Then there are tensor (not merely $\k$-linear abelian) equivalences and the last two categories are related by a tensor isomorphism:
\begin{equation}\label{eqn: 1.2}
\big({}^K_K\M^ K_H,\;\square_K\big)^\vee
\approx\big({}^{K^\ast}_{K^\ast}\M^{ H^\ast}_{K^\ast},\;\otimes_{K^\ast}\big)
\approx\Rep(K^{\ast\cop}\bowtie_\sigma H)\cong{}_H\mathfrak{YD}^K,
\end{equation}
where
\begin{itemize}
\item
The categories ${}^K_K\M^ K_H$ and ${}^{K^\ast}_{K^\ast}\M^{ H^\ast}_{K^\ast}$ consist of two-sided two-cosided ``relative'' Hopf modules induced by the Hopf algebra maps $\sigma_r$ and $\sigma_l$ respectively, and $(-)^\vee$ denotes the category with reversed arrows;
\item
The category ${}_H\mathfrak{YD}^K$ consists of ``relative'' (left-right) Yetter-Drinfeld modules induced by Hopf algebra map $\sigma_r$ (as a special situation of crossed modules introduced in \cite{CMZ97}).
\end{itemize}
Detailed structures of these categories may be found in Subsections \ref{subsection 2.2} and \ref{subsection 4.1}.
\end{theorem}

We should remark that if we only focus on (\ref{eqn: 1.2}) as $\k$-linear abelian equivalences, then some of them can become particular cases of other known results. They also generalize Schauenburg's characterization ${}^H\YD_H\approx{}_H^H\M{}_H^H$ of $\k$-linear abelian categories: In 1998, Beattie, D\u{a}sc\u{a}lescu, Raianu and Van Oystaeyen \cite{BDRV98} established
${}^C\YD_A\approx{}^C_H\M_A^H$ for any $H$-bimodule coalgebra $C$ and $H$-bicomodule algebra $A$. It was furthermore generalized by Schauenburg \cite{Sch99} in 1999 to an equivalence from a category denoted by ${}_R^D\M_T^H$ (with ``four distinct angles'').
Of course,
there are quasi-Hopf algebra versions of these results as well (e.g. \cite{BC03,BT06}), as well as categorical versions (e.g. \cite{BT21}).

Another notion closely related to our results is referred to as the relative center category (e.g. \cite{GNN09}). Suppose that the Hopf algebra map $\sigma_r:H\rightarrow K$ induced by the pairing $\sigma$ is surjective. Then our main results in Theorems \ref{thm: 1.1} and \ref{thm: 1.2} can be applied, in order to explain in Proposition \ref{prop:relativecenterequiv} that the relative center $\Z_{\Rep(K)}(\Rep(H))$ is tensor equivalent to the categories listed in (\ref{eqn: 1.2}). Note in the literature that there are various ``relative center'' constructions which are useful, for instance \cite{Lau20}. Nevertheless, when $\sigma_r$ is not surjective, we do not know how to give a categorical definition analogous to the relative center, containing our descriptions for $\Rep(K^{\ast\cop}\bowtie_\sigma H)$ or ${}_H\mathfrak{YD}^K$ as a natural example (Question \ref{question}).

The paper is organized as follows:
In Section \ref{section:2}, we recall and introduce some concepts and their properties,
including the quantum double, (relative) Yetter-Drinfeld modules, as well as the left partially dualized quasi-Hopf algebras.
Section \ref{Section3} is devoted to realizing the quantum double $K^{\ast\cop}\bowtie_\sigma H$ as a left partial dual of $K^\op\otimes H$, and then provides the corresponding tensor equivalences according to its reconstruction, involving the observation of the relative center $\Z_{\Rep(K)}(\Rep(H))$ when $\sigma_r$ is surjective.
Finally in Section \ref{Section 4}, the structures of the tensor categories of two-sided two-cosided relative Hopf modules are considered, which are shown to be equivalent to $\Rep(K^{\ast\cop}\bowtie_\sigma H)$ and ${}_H\mathfrak{YD}^K$. We also explain why this result generalizes Schauenburg's characterization at last.

\section{Preliminaries}\label{section:2}

Throughout this paper, all vector spaces are assumed to be over a field $\k$, and the tensor product over $\k$ is denoted simply by $\otimes$.

We refer to \cite{Swe69, Mon93, Rad12} and \cite{EGNO15} for the definitions and basic properties about Hopf algebras and tensor categories respectively, and we always make the following
identifications of Hopf algebras via the canonical isomorphisms:
\begin{equation}\label{eqn:indentification}
(H^\ast)^\ast=H,\;\;\;\;(H^\op)^\ast= H^{\ast\cop},\;\;\;\;
(H^\cop)^\ast= H^{\ast\op}\;\;\;\;
\text{and}\;\;\;\;
(H\otimes K)^\ast= H^\ast\otimes K^\ast
\end{equation}
for any finite-dimensional Hopf algebras $H$ and $K$.

Moreover, for any finite-dimensional quasi-Hopf algebra $H$ (\cite{Dri89, Kas95}), we always denote the category of its (left) finite-dimensional representations by $\Rep(H)$, which is known to be canonically a finite tensor category.

\subsection{(Generalized) quantum doubles of Hopf algebras}

Some of the most important finite-dimensional Hopf algebras are Drinfeld or quantum doubles.
The quantum double construction is due to Drinfeld \cite{Dri86}.
It can be
regarded as a special case of the quantum double introduced by Doi and Takeuchi \cite{DT94} (see also \cite{Maj95}), which is defined via two Hopf algebras and a skew pairing between them.

In this paper, we will recall the definition of quantum doubles in the language of Hopf pairings. Specifically, let $A$ and $H$ be two Hopf algebras. Then a \textit{Hopf pairing} (e.g. \cite{Maj90}) between $A$ and $H$ is a bilinear form $\sigma: A\otimes H\rightarrow\k$ satisfying the following conditions
$$\begin{array}{ll}
\mathrm{(i)}\;\;\;\;\sigma(aa',h)=\sum\sigma(a,h_{(1)})\sigma(a',h_{(2)}),
& \;\;\mathrm{(ii)}\;\;\;\;\sigma(a,hh')=\sum\sigma(a_{(1)},h)\sigma(a_{(2)},h'),  \\
\mathrm{(iii)}\;\;\sigma(1,h)=\varepsilon(h),
& \;\;\mathrm{(iv)}\;\;\;\sigma(a,1)=\varepsilon(a),
\end{array}$$
for all $a,a'\in A$ and $h,h'\in H$.
Note that the conditions (i)-(iv) imply that
$$
\mathrm{(v)}\;\;\;\;\;\;\;\;\sigma(a,S(h))=\sigma(S(a),h)
$$
holds for all $a\in A$ and $h\in H$ as well (see for instance \cite[Remark 2.2(ii)]{BB08}).

Now let $H$ and $K$ be finite-dimensional Hopf algebras.
We will always use the Sweedler notation
to indicate the coproduct of elements in $H$, $K$, $H^\ast$ and $K^\ast$. Besides, the following standard notations induced by a Hopf pairing between $K^\ast$ and $H$ will be used frequently.

\begin{notation}\label{not:sigma lr}
Suppose $\sigma: K^\ast\otimes H\rightarrow\k$ is a Hopf pairing. Then there are canonical Hopf algebra maps
\begin{equation}\label{eqn: sigma lr}
\sigma_l: K^\ast\rightarrow H^\ast,\;\;k^\ast\mapsto\sigma(k^\ast, -)
\;\;\;\;\text{and}\;\;\;\;\sigma_r: H\rightarrow K,\;\;h\mapsto\sigma(-, h),
\end{equation}
satisfying $\sigma_l=\sigma_r^\ast$.
\end{notation}

For convenience in this paper, we will use the following formulation of the quantum double (of finite-dimensional Hopf algebras),
which is formulated using
a Hopf pairing instead of a skew one.

\begin{definition}(cf. \cite[Proposition 2.2]{DT94})
Let $H$ and $K$ be finite-dimensional Hopf algebras, and let $\sigma: K^\ast\otimes H\rightarrow\k$ be a Hopf pairing. Denote by
\begin{equation}\label{eqn: sigmabar}
\overline{\sigma}\textcolor{purple}{:}=\sigma\circ(\id_{K^{\ast\cop}}\otimes S_H^{-1})
\end{equation}
the convolution inverse of $\sigma$ in $\Hom_\k(K^{\ast\cop}\otimes H,\k)$.
The quantum double $K^{\ast\cop}\bowtie_\sigma H$ is a Hopf algebra, with $K^{\ast\cop}\otimes H$ as its underlying vector space.
The multiplication is given by
\begin{equation}\label{eqn: multiplication}
(k^\ast\bowtie h)(k'^\ast\bowtie h')=\sum \sigma(k'^\ast_{(3)}, h_{(1)})k^\ast k'^\ast_{(2)}\bowtie h_{(2)}h'\overline{\sigma}(k'^\ast_{(1)}, h_{(3)})
\end{equation}
for all $k^\ast, k'^\ast \in K^{\ast\cop}$ and $h, h'\in H$, with identity element $\e\bowtie 1$;
The comultiplication is given by
\begin{equation}\label{eqn: comultiplication}
\Delta(k^\ast\bowtie h)=\sum (k^\ast_{(2)}\bowtie h_{(1)})\otimes(k^\ast_{(1)}\bowtie h_{(2)})
\end{equation}
for all $k^\ast\in K^{\ast\cop}$ and $h\in H$, with counit $1\otimes\e$.
The antipode of $K^{\ast\cop}\bowtie_\sigma H$ is given by
$$
S(k^\ast\bowtie h)=(1\bowtie S_H(h))(S_{K^\ast}^{-1}(k^\ast)\bowtie 1)
$$
for all $k^\ast\in K^{\ast\cop}$ and $h\in H$.
\end{definition}

It is clear by \cite[Remark 2.3]{DT94} that when $K=H$ and $\sigma$ is the evaluation, the quantum double $H^{\ast\cop}\bowtie_\sigma H$ is in fact the Drinfeld double $D(H)$ of $H$. See
\cite[Section 10.3]{Mon93} for the definition.

\subsection{Relative Yetter-Drinfeld modules and some canonical equivalences}\label{subsection 2.2}

Let $H$ be a finite-dimensional Hopf algebra. It is known that there are four ``kinds'' of categories
\begin{equation}\label{eqn:4 YD}
{}_H\mathfrak{YD}^H,\;\;{}^H\mathfrak{YD}_H,\;\;{}^H_H\mathfrak{YD}
\;\;\text{and}\;\;\mathfrak{YD}{}^H_H
\end{equation}
of Yetter-Drinfeld modules over $H$ introduced in the literature, see \cite[Section 3]{RT93} for example. They consist respectively of objects which are both $H$-modules and $H$-comodules with certain compatibility conditions.

In this paper, for any Yetter-Drinfeld module $V$ over $H$, we use angle brackets to express the (left or right) $H$-coaction on $v\in V$ as follows:
\begin{equation}
v\mapsto\sum v_{\langle-1\rangle}\otimes v_{\langle0\rangle}\in H\otimes V\;\;\;\text{or}\;\;\;
v\mapsto\sum v_{\langle0\rangle}\otimes v_{\langle1\rangle}\in V\otimes H.
\end{equation}

\begin{lemma}(\cite{Maj91})\label{lem: HYDH cong D(H)}
Let $H$ be a finite-dimensional Hopf algebra. Then there is an isomorphism
\begin{equation}
{}_H\mathfrak{YD}^H\cong \Rep(D(H))
\end{equation}
of braided finite tensor categories. Specifically, for each object $V\in{}_H\mathfrak{YD}^H$, the left $D(H)$-action on $V$ is defined by
\begin{equation}
(f\bowtie h)\cdot v=\sum (h\cdot v)_{\langle0\rangle}\langle f, (h\cdot v)_{\langle1\rangle}\rangle
\end{equation}
for all $f\in H^{\ast\cop}$, $h\in H$ and $v\in V$.
\end{lemma}

Now we consider a ``relative version'' of Yetter-Drinfeld modules over a Hopf pairing $\sigma: K^\ast\otimes H\rightarrow \k$ for later use. This can be a particular case of the notion introduced in \cite[Section 4]{HJ13}, as well as a situation of crossed $(H,H,K)$-modules defined in \cite[Section 2]{CMZ97}.

\begin{definition}\label{def: YD}
Let $H$ and $K$ be finite-dimensional Hopf algebras with Hopf pairing $\sigma: K^\ast\otimes H\rightarrow \k$.
\begin{itemize}
  \item [\rm{(1)}] The category ${}_H\mathfrak{YD}^K$ consists of finite-dimensional vector spaces $V$ which are both left $H$-modules and right $K$-comodules, such that the following compatibility condition holds:
\begin{equation}\label{eqn: lHmodrKcomod YD}
\sum(h\cdot v)_{\langle0\rangle}\otimes(h\cdot v)_{\langle1\rangle}
=\sum (h_{(2)}\cdot v_{\langle0\rangle})\otimes\sigma_r(h_{(3)})
 v_{\langle1\rangle}S^{-1}(\sigma_r(h_{(1)}))
\end{equation}
for all $h\in H$ and $v\in V$.
  \item [\rm{(2)}] The category ${}^K\mathfrak{YD}_H$ consists of finite-dimensional vector spaces $V$ which are both right $H$-modules and left $K$-comodules, such that the following compatibility condition holds:
\begin{equation}\label{eqn: lKcomodrHmod of YD}
\sum(v\cdot h)_{\langle-1\rangle}\otimes(v\cdot h)_{\langle0\rangle}=\sum S^{-1}(\sigma_r(h_{(3)}))v_{\langle-1\rangle}\sigma_r(h_{(1)})
\otimes (v_{\langle0\rangle}\cdot h_{(2)})
\end{equation}
for all $h\in H$ and $v\in V$.
\end{itemize}
\end{definition}

Furthermore, similarly to the fact that the categories (\ref{eqn:4 YD}) are finite tensor categories, we have the following lemma.

\begin{lemma}\label{lem: tps of YD mod}
Let $H$ and $K$ be finite-dimensional Hopf algebras with Hopf pairing $\sigma: K^\ast\otimes H\rightarrow \k$. Then ${}_H\mathfrak{YD}^K$ and ${}^K\mathfrak{YD}_H$ are both finite tensor categories. Specifically:
\begin{itemize}
\item[\rm{(1)}]
For any $V, W\in{}_H\mathfrak{YD}^K$, their tensor product is defined to be $V\otimes W$ with left $H$-module structure
\begin{equation}\label{eqn: lH YD}
h\otimes(v\otimes w)\mapsto\sum (h_{(1)}\cdot v)\otimes (h_{(2)}\cdot w)
\;\;\;\;\;\;\;\;(h\in H,\;v\in V,\;w\in W)
\end{equation}
and right $K$-comodule structure
\begin{equation}\label{eqn: rK YD}
v\otimes w\mapsto\sum (v_{\langle0\rangle}\otimes w_{\langle0\rangle})\otimes w_{\langle1\rangle}v_{\langle1\rangle}\;\;\;\;\;\;\;\;(v\in V,\;w\in W).
\end{equation}

\item[\rm{(2)}]
For any $V, W\in {}^K\mathfrak{YD}_H$, their tensor product is defined to be $V\otimes W$ with right $H$-module structure
$$
(v\otimes w)\otimes h\mapsto \sum (v\cdot h_{(1)})\otimes (w\cdot h_{(2)})
\;\;\;\;\;\;\;\;(h\in H,\;v\in V,\;w\in W)
$$
and left $K$-comodule structure
$$
v\otimes w\mapsto \sum w_{\langle-1\rangle}v_{\langle-1\rangle}\otimes
(v_{\langle0\rangle}\otimes w_{\langle0\rangle})
\;\;\;\;\;\;\;\;(v\in V,\;w\in W).
$$
\end{itemize}
\end{lemma}

Under the assumptions of Lemma \ref{lem: tps of YD mod}, there is another category ${}_{K^\ast}\YD^{H^\ast}$ of (left-right) Yetter-Drinfeld modules in the sense of Definition \ref{def: YD}(1) with Hopf pairing
$$\sigma':H\otimes K^\ast\rightarrow\k,\;\;\;\;h\otimes k^\ast\mapsto\sigma(k^\ast,h),$$
where $\sigma'_l=\sigma_r$ and $\sigma'_r=\sigma_l$ hold in this situation.

In this paper, we will concentrate on $\big({}_{K^\ast}\mathfrak{YD}^{H^\ast}\big)^\rev$, which denotes the finite tensor category with reverse tensor products to ${}_{K^\ast}\mathfrak{YD}^{H^\ast}$.
One can find that
for any objects $V,W\in\big({}_{K^\ast}\mathfrak{YD}^{H^\ast}\big)^\rev$, their tensor product $V\otimes W$ will have the left $K^\ast$-module structure
\begin{equation}\label{eqn: lK* *YD*}
k^\ast\otimes(v\otimes w)\mapsto\sum (k^\ast_{(2)}\cdot v)\otimes (k^\ast_{(1)}\cdot w)
\;\;\;\;\;\;\;\;(k^\ast\in K^\ast,\;v\in V,\;w\in W)
\end{equation}
and right $H^\ast$-comodule structure
\begin{equation}\label{eqn: rH* *YD*}
v\otimes w\mapsto\sum (v_{\langle0\rangle}\otimes w_{\langle0\rangle})\otimes v_{\langle1\rangle}w_{\langle1\rangle}\;\;\;\;\;\;\;\;(v\in V,\;w\in W).
\end{equation}
In fact, this tensor category is indeed isomorphic to ${}_H\mathfrak{YD}^K$.

\begin{proposition}\label{prop: YD cong *YD*}
Let $H$ and $K$ be finite-dimensional Hopf algebras with Hopf pairing $\sigma: K^\ast\otimes H\rightarrow\k$. Then there is an isomorphism of tensor categories
\begin{equation}\label{eqn: YD cong *YD*}
{}_H\mathfrak{YD}^K\cong \big({}_{K^\ast}\mathfrak{YD}^{H^\ast}\big)^\rev,
\end{equation}
which sends $V\in{}_H\mathfrak{YD}^K$ to the vector space $V$ with left $K^\ast$-module structure $\rightharpoonup$ defined by
\begin{equation}\label{lK*rH* YD}
k^\ast\rightharpoonup v=\sum v_{\la0\ra}\la k^\ast,v_{\la1\ra}\ra
\;\;\;\;(\forall k^\ast\in K^\ast,\;\forall v\in V),
\end{equation}
as well as the right $H^\ast$-comodule structure
\begin{equation}\label{eqn: rH*}
v\mapsto\sum v^{\la0\ra}\otimes v^{\la1\ra}
\;\;\;\;\text{such that}\;\;\;\;
\sum v^{\la0\ra}\la v^{\la1\ra},h\ra=h\cdot v\;\;\;\;(\forall h\in H).
\end{equation}
\end{proposition}

\begin{proof}
This follows by a direct computation using the compatibility condition (\ref{eqn: lHmodrKcomod YD}) in ${}_H\mathfrak{YD}^K$.
\end{proof}

\subsection{Partially admissible mapping systems and left partial dualizations}

The notion of \textit{left partially dualized quasi-Hopf algebra} was introduced in \cite{Li23}, which reconstructs a certain form of dual tensor category.
Let us recall the notion of \textit{partially admissible mapping system}.

\begin{definition}(\cite[Definition 2.6]{Li23})\label{def:PAMS}
Let $H$ be a finite-dimensional Hopf algebra. Suppose that
\begin{itemize}
\item[\rm{(1)}]
$\iota:B\rightarrowtail H$ is an injection of left $H$-comodule algebras, and
$\pi:H\twoheadrightarrow C$ is a surjection of right $H$-module coalgebras;
\item[\rm{(2)}]
The image of $\iota$ equals the space of the coinvariants of the right $C$-comodule $H$ with structure $(\mathrm{id}_H\otimes\pi)\circ\Delta$.
\end{itemize}
Then the pair of $\k$-linear diagrams
\begin{equation*}
\begin{array}{ccc}
\xymatrix{
B \ar@<.5ex>[r]^{\iota} & H \ar@<.5ex>@{-->}[l]^{\zeta} \ar@<.5ex>[r]^{\pi}
& C \ar@<.5ex>@{-->}[l]^{\gamma}  }
&
\;\;\text{and}\;\;
&
\xymatrix{
C^\ast \ar@<.5ex>[r]^{\pi^\ast}
& H^\ast \ar@<.5ex>@{-->}[l]^{\gamma^\ast} \ar@<.5ex>[r]^{\iota^\ast}
& B^\ast \ar@<.5ex>@{-->}[l]^{\zeta^\ast}  },
\end{array}
\end{equation*}
is said to be a {\pams} for $\iota$, denoted by $(\zeta,\gamma^\ast)$ for simplicity, if all the conditions
\begin{itemize}
\item[\rm{(3)}]
$\zeta$ and $\gamma$ have convolution inverses $\overline{\zeta}$ and $\overline{\gamma}$ respectively;
\item[\rm{(4)}]
$\zeta$ preserves left $B$-actions, and $\gamma$ preserves right $C$-coactions;
\item[\rm{(5)}]
$\zeta$ and $\gamma$ preserve both the units and counits, meaning that
$$\zeta(1_H)=1_B,\;\;\;\;\e_B\circ\zeta=\e_H,\;\;\;\;
\gamma(1_C)=1_H\;\;\;\;\text{and}\;\;\;\;\e_H\circ\gamma=\e_C,$$
where we make the
convention $1_C:=\pi(1_H)$ and $\e_B:=\iota^\ast(\e_H)$;
\item[\rm{(6)}]
$(\iota\circ\zeta)\ast(\gamma\circ\pi)=\id_H$,
\end{itemize}
and equivalently, the dual forms of \rm{(1)} to \rm{(6)} hold.
\end{definition}

Some elementary properties of partially admissible mapping systems should be mentioned for later
use.

\begin{lemma}(\cite[Proposition 2.9 (1) and (2)]{Li23})\label{lem:PAMS properties}
Suppose that $(\zeta, \gamma^\ast)$ is a partially admissible mapping system for $\iota: B\rightarrowtail H$. Then:
\begin{itemize}
\item[\rm{(1)}]
$\iota\circ\zeta=\id_B$ and $\pi\circ\gamma=\id_C$ and their dual forms hold;
\item[\rm{(2)}]
$\zeta\circ\gamma=\langle\e_C, -\rangle1_B$
as linear maps from $C$ to $B$, where the notation $\langle\e_C, -\rangle1_B$ denotes the product of the evaluation morphism $\langle\e_C, -\rangle$ and the unit element $1_B$.
\end{itemize}
\end{lemma}

Evidently, the right $H$-module coalgebra surjection $\pi:H\twoheadrightarrow C$ induces  to the injection $\pi^\ast:C^\ast\rightarrowtail H^\ast$ of right $H^\ast$-comodule algebras.
We will use notations similar with \cite[Section 2.1]{Li23}
that
$$\begin{array}{ccc}
{\begin{array}{ccc}
B &\rightarrow& H\otimes B \\
b &\mapsto& \sum b_{(1)}\otimes b_{(2)}
\end{array}}
&\;\;\; \text{and}\;\;\;&
{\begin{array}{ccc}
C^\ast &\rightarrow& C^\ast\otimes H^\ast \\
x^\ast &\mapsto& \sum x^\ast_{(1)}\otimes x^\ast_{(2)}
\end{array}}
\end{array}$$
to represent the structures of the left $H$-comodule $B$ and the right $H^\ast$-comodule $C^\ast$ respectively.
Furthermore, we denote
\begin{equation}\label{hit action}
b\leftharpoonup h^\ast:= \sum \langle h^\ast, b_{(1)}\rangle b_{(2)}
\;\;\;\;\;\;\text{and}\;\;\;\;\;\;
h\rightharpoonup x^\ast:= \sum x^\ast_{(1)}\langle x^\ast_{(2)}, h\rangle
\end{equation}
for any $h^\ast\in H^\ast$, $b\in B$ and $h\in H$, $x^\ast\in C^\ast$. It is clear that $(B,\leftharpoonup)$ is a right $H^\ast$-module and $(C^\ast,\rightharpoonup)$ is a left $H$-module.
However, the left and right hit actions of $H^\ast$ on $H$ (or vice versa) are also denoted by $\rightharpoonup$ and $\leftharpoonup$ without confusions.

With the above notation,
the smash product algebra $C^\ast\# B$ has a quasi-Hopf algebra structure with associator $\pd{\phi}$ given by the inverse of the element
\begin{equation}\label{eqn:phi^-1}
\pd{\phi}^{-1}=\sum_{i,j}\left(\e\#\zeta[\gamma(x_i)
  \gamma(x_j)_{(1)}]\right)
  \otimes\left(x_i^\ast\#\zeta[\gamma(x_j)_{(2)}]\right)
  \otimes\left(x_j^\ast\#1\right),
\end{equation}
where $\{x_i\}$ is a linear basis of $C$ with dual basis $\{x_i^\ast\}$ of $C^\ast$. It is called in \cite{Li23} a \textit{left partially dualized quasi-Hopf algebra} (or \textit{left partial dual}) of $H$ determined by $(\zeta, \gamma^\ast)$. The whole quasi-Hopf algebra structures of the left partial dual $C^\ast\# B$ can be found in \cite[Section 3.1]{Li23}.

It is known that the quasi-Hopf algebra $C^\ast\# B$
becomes a Hopf algebra
when its associator $\pd{\phi}$ (or its inverse $\pd{\phi}^{-1}$) is trivial. In this case, we also say that $C^\ast\# B$ is a \textit{left partially dualized Hopf algebra} of $H$.
The following lemma states a sufficient condition for this situation, and some others can be found in \cite[Section 6]{Li23}.

\begin{lemma}\label{lem: lpd is Hopf alg}
Let $H$, $B$ and $C$ be finite-dimensional Hopf algebras. Suppose the algebra $B$ is a left $H$-comodule algebra, and the coalgebra $C$ is a right $H$-module coalgebra, satisfying that
\begin{equation*}
\begin{array}{ccc}
\xymatrix{
B \ar@<.5ex>[r]^{\iota} & H \ar@<.5ex>@{-->}[l]^{\zeta} \ar@<.5ex>[r]^{\pi}
& C \ar@<.5ex>@{-->}[l]^{\gamma}  }
&
\;\;\text{and}\;\;
&
\xymatrix{
C^\ast \ar@<.5ex>[r]^{\pi^\ast}
& H^\ast \ar@<.5ex>@{-->}[l]^{\gamma^\ast} \ar@<.5ex>[r]^{\iota^\ast}
& B^\ast \ar@<.5ex>@{-->}[l]^{\zeta^\ast}  },
\end{array}
\end{equation*}
is a partially admissible mapping system for $\iota$.
If $\zeta$ and $\gamma$ are Hopf algebra maps, then the left partial dual $C^\ast\#B$ determined by $(\zeta,\gamma^\ast)$ is a Hopf algebra, and its coalgebra structure is the tensor product $C^\ast\otimes B$.
\end{lemma}

\begin{proof}
Suppose $\{x_i\}$ is a linear basis of $C$ with dual basis $\{x_i^\ast\}$ of $C^\ast$ as usual.

In order to show that $C^\ast\#B$ is a left partially dualized Hopf algebra, it suffices to verify that the inverse $\pd{\phi}^{-1}$ of its associator is trivial. In fact,
since $\gamma$ is a bialgebra map, we compute that
\begin{eqnarray*}
\pd{\phi}^{-1}
&\overset{(\ref{eqn:phi^-1})}{=}&
\sum_{i,j}\left(\e\#\zeta[\gamma(x_i)
\gamma(x_j)_{(1)}]\right)
\otimes\left(x_i^\ast\#\zeta[\gamma(x_j)_{(2)}]\right)
\otimes\left(x_j^\ast\#1\right) \\
&=&
\sum_{i,j}\left(\e\#\zeta[\gamma(x_i
x_j{}_{(1)})]\right)
\otimes\left(x_i^\ast\#\zeta[\gamma(x_j{}_{(2)})]\right)
\otimes\left(x_j^\ast\#1\right) \\
&\overset{\text{Lemma \ref{lem:PAMS properties}(2)}}{=}&
\sum_{i,j}\left(\e\#\la\e, x_i x_j{}_{(1)}\ra 1\right)
\otimes\left(x_i^\ast\#\la\e, x_j{}_{(2)}\ra 1\right)
\otimes\left(x_j^\ast\#1\right) \\
&=&
(\e\#1)\otimes(\e\#1)\otimes(\e\#1).
\end{eqnarray*}

Moreover, we have the following computations for the ``comultiplication'' $\pd{\Delta}$ on the left partial dual $C^\ast\#B$: Note that $\gamma$ is a coalgebra map, and $\zeta$ is an algebra map. Thus for every $x^\ast\in C^\ast$, we know by \cite[Remark 3.4(2)]{Li23} that
\begin{eqnarray*}
\pd{\Delta}(x^\ast\#1)
&=&
\sum_{i}\left(x^\ast_{(1)}\#\zeta[\gamma(x_i)\leftharpoonup x^\ast_{(2)}]\right)\otimes\left(x^\ast_i\#1\right)  \\
&\overset{\text{Lemma \ref{lem:PAMS properties}(2)}}{=}&
\sum_{i}(x^\ast_{(1)}\#\langle x^\ast_{(2)}, \gamma(x_i{}_{(1)})\rangle\la\e, x_i{}_{(2)}\ra1)
\otimes(x^\ast_i\#1) \\
&=&
\sum_{i}(x^\ast_{(1)}\#\langle \gamma^\ast(x^\ast_{(2)}),x_i\rangle1)
\otimes(x^\ast_i\#1)  \\
&=&
\sum(x^\ast_{(1)}\#1)\otimes
(\gamma^\ast(x^\ast_{(2)})\#1),
\end{eqnarray*}
Similarly, for every $b\in B$, by \cite[Theorem 3.1(1)]{Li23} and Lemma \ref{lem:PAMS properties}(2), we have
\begin{eqnarray*}
\pd{\Delta}(\e\#b)
&=&
\sum_{i}\left(\e\#\zeta(b_{(1)})\right)\otimes(\e\#b_{(2)}).
\end{eqnarray*}
Consequently, 
\begin{eqnarray}\label{eqn:Delta(x* b)}
\pd{\Delta}(x^\ast\#b)
&=& \pd{\Delta}(x^\ast\#1)\pd{\Delta}(\e\#b)  \nonumber  \\
&=& \sum(x^\ast_{(1)}\#\zeta(b_{(1)}))\otimes
(\gamma^\ast(x^\ast_{(2)})\#b_{(2)})
\;\;\;\;\;\;\;\;(\forall x^\ast\in C^\ast,\;\forall b\in B).\;\;\;\;
\end{eqnarray}

Now let us show that the Hopf algebra $B$ has comultiplication $\Delta_B:b\mapsto\sum\zeta(b_{(1)})\otimes b_{(2)}$. Since $\zeta$ is assumed to be a
coalgebra map, we know for each $b\in B$ that
\begin{eqnarray*}
\Delta_B(b)
&\overset{\text{Lemma \ref{lem:PAMS properties}(1)}}{=}&
\Delta_B(\zeta[\iota(b)])
~=~ (\zeta\otimes\zeta)\circ\Delta(\iota(b))  \\
&=& \sum \zeta[\iota(b)_{(1)}]\otimes\zeta[\iota(b)_{(2)}]
~=~ \sum\zeta(b_{(1)})\otimes b_{(2)},
\end{eqnarray*}
where the last equality is because $\iota$ is a left $H$-comodule map by Definition \ref{def:PAMS}(1).

Similarly, one could also find that the Hopf algebra $C^\ast$ has comultiplication $x^\ast\mapsto \sum x^\ast_{(1)}\otimes\gamma^\ast(x^\ast_{(2)})$.
In conclusion,
the left partial dual $C^\ast\#B$ is the tensor product $C^\ast\otimes B$ as a coalgebra with comultiplication (\ref{eqn:Delta(x* b)}) and counit
$x^\ast\#b\mapsto\la x^\ast,1_C\ra\la\e_B,b\ra$.
\end{proof}

At the end of this subsection, we introduce the tensor equivalences
\begin{equation}\label{eqn:Doi}
\Rep(C^\ast\#B)\approx {}^{}_{C^\ast}\M^{H^\ast}_{C^\ast}
\approx \Rep(H)_{\Rep(B)}^\ast,
\end{equation}
which can be regarded as the \textit{reconstruction theorem for left partial duals} by \cite[Section 4]{Li23}.
Here:
\begin{itemize}
\item
$\Rep(B)$ is regarded as the left $\Rep(H)$-module category in the sense of \cite[Section 1.5]{AM07} determined by the left $H$-comodule algebra injection $\iota:B\rightarrowtail H$. Specifically, the module product bifunctor
$\ogreaterthan:\Rep(H)\times \Rep(B)\rightarrow \Rep(B)$ sends each object $(X,V)$ to the space $X\otimes_\k V$ with $B$-action defined via
\begin{equation}\label{eqn:Rep(H)modRep(B)}
b\cdot(x\otimes v)=\sum b_{(1)}x\otimes b_{(2)}v\;\;\;\;\;\;\;\;
(\forall b\in B,\;\forall x\in X,\forall v\in V).
\end{equation}
Moreover,
$$\Rep(H)_{\Rep(B)}^\ast:=\mathrm{Rex}_{\Rep(H)}(\Rep(B))^\mathrm{rev}$$
is the \textit{dual tensor category} consisting of right exact $\Rep(H)$-endofunctors on $\Rep(B)$, with tensor product bifunctor chosen as the \textit{opposite} to the composition;

\item
${}^{}_{C^\ast}\M^{H^\ast}_{C^\ast}$ is the category
of finite-dimensional \textit{relative Doi-Hopf modules}. Specifically, it consists of finite-dimensional $C^\ast$-$C^\ast$-bimodules $M$ equipped with right $H^\ast$-comodule structure preserving both left and right $C^\ast$-actions:
For any $m\in M$ and $x^\ast\in C^\ast$, the equations
\begin{equation}\label{eqn: rcomod of lx*mod}
\sum(x^\ast\cdot m)_{(0)}\otimes(x^\ast\cdot m)_{(1)}
= \sum x^\ast_{(1)}\cdot m_{(0)}\otimes x^\ast_{(2)} m_{(1)}
\in M\otimes H^\ast,
\end{equation}
\begin{equation}\label{eqn: rcomod of rx*mod}
\sum(m\cdot x^\ast)_{(0)}\otimes(m\cdot x^\ast)_{(1)}
= \sum m_{(0)}\cdot x^\ast_{(1)} \otimes m_{(1)} x^\ast_{(2)}
\in H^\ast\otimes M
\end{equation}
hold, where $m\mapsto\sum m_{(0)}\otimes m_{(1)}$ denotes the right $H^\ast$-comodule structure on $M$. It is mentioned in \cite[Section 4.2]{Li23} that ${}^{}_{C^\ast}\M^{H^\ast}_{C^\ast}$ is a finite tensor category.
\end{itemize}

\begin{lemma}(\cite[Theorem 4.22]{Li23})\label{lem: reconstructionthm}
Let $H$ be a finite-dimensional Hopf algebra. Suppose that
\begin{equation*}
\begin{array}{ccc}
\xymatrix{
B \ar@<.5ex>[r]^{\iota} & H \ar@<.5ex>@{-->}[l]^{\zeta} \ar@<.5ex>[r]^{\pi}
& C \ar@<.5ex>@{-->}[l]^{\gamma}  }
&
\;\;\text{and}\;\;
&
\xymatrix{
C^\ast \ar@<.5ex>[r]^{\pi^\ast}
& H^\ast \ar@<.5ex>@{-->}[l]^{\gamma^\ast} \ar@<.5ex>[r]^{\iota^\ast}
& B^\ast \ar@<.5ex>@{-->}[l]^{\zeta^\ast}  },
\end{array}
\end{equation*}
is a partially admissible mapping system $(\zeta, \gamma^\ast)$.
Then there is a tensor equivalence $\Phi$ between
\begin{itemize}
  \item [\rm{(1)}] The category ${}^{}_{C^\ast}\M^{H^\ast}_{C^\ast}$ of finite-dimensional relative Doi-Hopf modules, and
  \item [\rm{(2)}] The category of finite-dimensional representations of the left partial dual $C^\ast\#B$ determined by $(\zeta, \gamma^\ast)$,
\end{itemize}
defined as
$$\begin{array}{cccc}
\Phi: & {}^{}_{C^\ast}\M^{H^\ast}_{C^\ast} &\approx& \Rep(C^\ast\#B), \\
&M&\mapsto& \overline{M}=M/M(C^\ast)^+,
\end{array}$$
with monoidal structure
$$
\begin{array}{cccc}
J_{M,N}: & \overline{M}\otimes\overline{N}&\cong&
\overline{M\otimes_{C^\ast}N} \\
&\overline{m}\otimes\overline{n} &\mapsto&
\sum\overline{m_{(0)}\overline{\gamma}^\ast(m_{(1)})\otimes_{C^\ast}n},
\end{array}
$$
where $(C^\ast)^+$ denotes the preimage of $\pi^\ast(C^\ast)\cap\ker(\e_{H^\ast})$ under the injection $\pi^\ast$.
\end{lemma}

\begin{remark}
Indeed, the equivalence $\Phi$ is the same as the functor provided in \cite[Section 1]{Tak79}.
\end{remark}

\section{Realization of the quantum double as left partial dual, and consequences}\label{Section3}

For the remainder
of this paper, let $H$ and $K$ be finite-dimensional Hopf algebras with a Hopf pairing $\sigma: K^\ast\otimes H \rightarrow \k$.
Recall from
Notation \ref{not:sigma lr} that there exist Hopf algebra maps
$$\sigma_l: K^\ast\rightarrow H^\ast,\;\;k^\ast\mapsto\sigma(k^\ast, -)
\;\;\;\;\text{and}\;\;\;\;\sigma_r: H\rightarrow K,\;\;h\mapsto\sigma(-, h).$$

\subsection{Quantum double as a left partial dual of the tensor product Hopf algebra}

Our main goal in this subsection is to show that the quantum double $K^{\ast\cop}\bowtie_\sigma H$ is a left partially dualized Hopf algebra of $K^\op\otimes H$.

\begin{lemma}\label{lem:PAMSleft0}
\begin{itemize}
\item[\rm{(1)}]
The algebra $H$ is a left $K^\op\otimes H$-comodule algebra via coaction
\begin{equation}\label{eqn:l Kop tensor H ms}
\rho: H\rightarrow(K^\op\otimes H)\otimes H,\;\; h\mapsto\sum(\sigma_r(S^{-1}(h_{(3)}))\otimes h_{(1)})\otimes h_{(2)};
\end{equation}
The coalgebra
$K^\op$ is a right $K^\op\otimes H$-module coalgebra via action
\begin{equation}\label{eqn:r Kop tensor H ms}
\btl: K^\op\otimes(K^\op\otimes H)\rightarrow K^\op,\;\; l\otimes(k\otimes h)\mapsto kl\sigma_r (h).
\end{equation}

\item[\rm{(2)}]
With structures defined in (1),
\begin{equation}\label{def: iota}
\iota: H\rightarrow K^\op\otimes H,\;\;
h\mapsto\sum \sigma_r(S^{-1}(h_{(2)}))\otimes h_{(1)}.
\end{equation}
is a map of left $K^\op\otimes H $-comodule algebras,
and
\begin{equation}\label{def: pi}
\pi: K^\op\otimes H\rightarrow K^\op, \;\;k\otimes h\mapsto k\sigma_r(h).
\end{equation}
is a map of right $K^\op\otimes H$-module coalgebras.

\item[\rm{(3)}]
With notation in (1), the image of $\iota$ equals the space of the coinvariants of the right $K^\op$-comodule $K^\op\otimes H$ with structure $(\mathrm{id}_H\otimes\pi)\circ\Delta$.
\end{itemize}
\end{lemma}

\begin{proof}
\begin{itemize}
\item[(1)]
These claims can be verified by direct computations, but here we explain how the structures arise from regular ones.

Let us show that $\rho$ is a left $K^\op\otimes H$-comodule structure on $H$ at first:
Consider the regular $H$-$H$-bicomodule structure on $H$, which is known to be equivalent to a left $H^\cop\otimes H$-comodule structure
\begin{equation}\label{eqn:l Hcop tensor H ms}
H\rightarrow (H^\cop\otimes H)\otimes H,\;\;\;\;
h\mapsto \sum (h_{(3)}\otimes h_{(1)})\otimes h_{(2)}.
\end{equation}
Furthermore, note that $\sigma_r\circ S^{-1}:H^\cop\rightarrow K^\op$ is a coalgebra map. Thus it induces from (\ref{eqn:l Hcop tensor H ms}) a left $K^\op\otimes H$-comodule structure on $H$, which is exactly $\rho$ defined in (\ref{eqn:l Kop tensor H ms}).

On the other hand,
since (\ref{eqn:l Hcop tensor H ms}) and $\sigma_r\circ S^{-1}:H^\cop\rightarrow K^\op$ are both algebra maps,
we conclude that $\rho$ is also an algebra map. This means that $H$ is a left $K^\op\otimes H$-comodule algebra via the comodule structure $\rho$.

Next, we show that $\btl$ is a right $K^{\op}\otimes H$-module structure on $K^\op$: Consider the free left and right $K$-module structures on $K^\op$ defined by the multiplication on $K$, and they make $K^\op$ become a $K$-$K$-bimodule. It is equivalent to a right $K^\op\otimes H$-module structure
\begin{equation}\label{eqn:r Kop tensor K ms}
K^\op\otimes(K^\op\otimes K)\rightarrow K^\op,\;\;\;\; l\otimes(k\otimes k')\mapsto klk'.
\end{equation}
Using the algebra map $\sigma_r:H\rightarrow K$, we see that
$K^\op$ admits a right $K^\op\otimes H$-module structure $\btl$.

Finally, note that $\sigma_r:H\rightarrow K$ and (\ref{eqn:r Kop tensor K ms}) are both coalgebra maps.
This implies that $\btl$ is also a coalgebra map,
and hence $K^{\op}$ is a right $K^{\op}\otimes H$-module coalgebra via the module structure $\btl$.

\item[(2)]
Let us verify that $\iota$ defined in (\ref{def: iota}) preserves left $K^{\op}\otimes H$-coactions, where the left $K^{\op}\otimes H$-comodule structure on $H$ is $\rho$. Indeed, it is straightforward to find that $\iota=(\id_{K^\op\otimes H}\otimes\e)\circ\rho$ holds, and hence
\begin{eqnarray*}
    \Delta_{K^\op\otimes H}\circ\iota
    &=& \Delta_{K^\op\otimes H}\circ(\id_{K^\op\otimes H}\otimes\e)\circ\rho  \\
    &=& (\id_{K^\op\otimes H}\otimes\id_{K^\op\otimes H}\otimes\e)\circ(\Delta_{K^\op\otimes H}\otimes\id_H)\circ\rho  \\
    &=& (\id_{K^\op\otimes H}\otimes\id_{K^\op\otimes H}\otimes\e)\circ
    (\id_{K^\op\otimes H}\otimes\rho)\circ\rho  \\
    &=& (\id_{K^\op\otimes H}\otimes\iota)\circ\rho,
\end{eqnarray*}
where the third equality is because $\rho$ is a left $K^\op\otimes H$-comodule structure.

Besides, we know in (1) that $\rho$ is an algebra map, which implies that $\iota=(\id_{K^\op\otimes H}\otimes\e)\circ\rho$ is also an algebra map.
In conclusion, $\iota$ is a map of left $K^\op\otimes H $-comodule algebras.

On the other hand, it is clear that $\pi$ (\ref{def: pi}) is a map of right $K^\op\otimes H$-modules. 
Moreover, one can also compute directly to prove $$\Delta_{K^{\op}}\circ\pi=(\pi\otimes\pi)\circ\Delta_{K^{\op}\otimes H}
\;\;\;\;\text{and}\;\;\;\;
\e_{K^{\op}}\circ\pi=\e_{K^{\op}\otimes H}$$
according to the fact that $\sigma_r$ is a coalgebra map.
Thus, $\pi$ is a map of right $K^{\op}\otimes H$-module coalgebras.

\item[(3)]
This follows from
combining the coopposite version of \cite[Proposition 3.10]{Mas94} as well as a fact in \cite[Theorem 6.1]{Skr07} that a finite-dimensional Hopf algebra must be cocleft over its left coideal subalgebra. However, we provide here a simpler proof instead:

It is direct to compute that
\begin{eqnarray*}
\sum\iota(h)_{(1)}\otimes\pi[\iota(h)_{(2)}] &=& \sum\left(\sigma_r(S^{-1}(h_{(4)}))\otimes h_{(1)}\right)\otimes\pi\left[\sigma_r(S^{-1}(h_{(3)}))\otimes h_{(2)}\right] \\
&=&
\sum\left(\sigma_r(S^{-1}(h_{(4)}))\otimes h_{(1)}\right)\otimes\sigma_r(S^{-1}(h_{(3)}))\sigma_r(h_{(2)}) \\
&=&
\sum\left(\sigma_r(S^{-1}(h_{(2)}))\otimes h_{(1)}\right)\otimes 1_{K^{\op}} ~=~ \iota(h)\otimes 1_{K^{\op}}
\end{eqnarray*}
holds for all $h\in H$, and hence
the image $\mathrm{Im}(\iota)$ is contained in the space $(K^\op\otimes H)_{\mathrm{coinv}}$ of the coinvariants.
Thus it suffices to show that $\dim(\mathrm{Im}(\iota))=\dim((K^\op\otimes H)_{\mathrm{coinv}})$.

In fact, one could verify that $K^\op\otimes H$ is a right $K^\op$-Hopf module with comodule structure $(\mathrm{id}_H\otimes\pi)\circ\Delta$ and module structure
$$
(K^\op\otimes H)\otimes K^\op\rightarrow K^\op\otimes H,\;\;\;(k\otimes h)\otimes l\mapsto lk\otimes h.
$$
Consequently, we know by the fundamental theorem of Hopf modules (\cite[Theorem 4.1.1]{Swe69}) that
$\dim(K^\op\otimes H)=\dim(K^\op)\dim((K^\op\otimes H)_{\mathrm{coinv}})$, which implies
$$\dim((K^\op\otimes H)_{\mathrm{coinv}})=\frac{\dim(K^\op\otimes H)}{\dim(K^\op)}
=\dim(H)=\dim(\mathrm{Im}(\iota)).$$
\end{itemize}
\end{proof}

Now we aim to construct a partially admissible mapping system $(\zeta, \gamma^\ast)$ for $\iota:H\rightarrowtail K^\op\otimes H$ defined in Lemma \ref{lem:PAMSleft0}(2).

\begin{lemma}\label{lem:PAMSleft}
notation in Lemma \ref{lem:PAMSleft0}, we have a partially admissible mapping system
\begin{equation}\label{eqn:admissiblemapsys}
\begin{array}{ccc}
\xymatrix{
H\ar@<.5ex>[r]^{\iota\;\;\;\;\;\;}
& K^\op\otimes H
 \ar@<.5ex>@{-->}[l]^{\zeta\;\;\;\;\;\;} \ar@<.5ex>[r]^{\;\;\;\;\;\;\pi}
& K^\op \ar@<.5ex>@{-->}[l]^{\;\;\;\;\;\;\gamma}  }
&\;\;\text{and}\;\;&
\xymatrix{
K^{\ast\cop}\ar@<.5ex>[r]^{\pi^\ast\;\;\;\;\;\;}
& K^{\ast\cop}\otimes H^\ast
 \ar@<.5ex>@{-->}[l]^{\gamma^\ast\;\;\;\;\;\;}
 \ar@<.5ex>[r]^{\;\;\;\;\;\;\;\;\iota^\ast}
& H^\ast \ar@<.5ex>@{-->}[l]^{\;\;\;\;\;\;\;\;\zeta^\ast}  }
\end{array}
\end{equation}
for $\iota$, where
\begin{equation}\label{zeta}
\zeta: K^\op\otimes H\rightarrow H, \;\; k\otimes h\mapsto\varepsilon(k)h
\end{equation}
  and
\begin{equation}\label{gamma}
  \gamma: K^\op\rightarrow K^\op\otimes H, \;\; k\mapsto k\otimes1.
\end{equation}
\end{lemma}

\begin{proof}
Our goal is to check the requirements for $(\zeta,\gamma^\ast)$ to be a partially admissible mapping system. Note that (1) and (2) in Definition \ref{def:PAMS} are confirmed in Lemma \ref{lem:PAMSleft0}, and we will check the conditions (3) to (6).

\begin{itemize}
\item[(3)]
It is straightforward to verify that the maps
$$\begin{array}{ccc}
{\begin{array}{cccc}
\overline{\zeta}: & K^\op\otimes H &\rightarrow& H  \\
& k\otimes h &\mapsto& \varepsilon(k)S(h)
\end{array}}
&\;\;\; \text{and}\;\;\;&
{\begin{array}{cccc}
\overline{\gamma}: & K^\op &\rightarrow& K^\op\otimes H  \\
& k &\mapsto& S^{-1}(k)\otimes 1
\end{array}}
\end{array}$$
are respectively convolution inverses of $\zeta$ and $\gamma$.

\item[(4)]
Let us verify that the map $\zeta$ defined in (\ref{zeta}) preserves left $H$-actions. Recall that the left $H$-module structure on $K^{\op}\otimes H$ is
$m_{K^{\op}\otimes H}\circ(\iota\otimes\id)$, that is,
\begin{equation}\label{cdot}
H\otimes(K^\op\otimes H) \rightarrow K^\op\otimes H,\;\;
h'\otimes(k\otimes h) \mapsto \sum k\sigma_r(S^{-1}(h'_{(2)}))\otimes h'_{(1)}h,
\end{equation}
and we have the following computation for any $h, h'\in H$ and $k\in K$,
\begin{eqnarray*}
\zeta(h'\cdot(k\otimes h))
&=&
\sum\zeta\big[k\sigma_r(S^{-1}(h'_{(2)}))\otimes h'_{(1)}h\big]
~=~ \e\big[k\sigma_r(S^{-1}(h'_{(2)}))\big] h'_{(1)}h  \\
&=& \e(k)h'h ~=~ h'\zeta(k\otimes h).
\end{eqnarray*}

On the other hand, it is straightforward to find that the map $\gamma$ defined in  (\ref{gamma}) preserves right $K^\op$-coactions, where the right $K^\op$-comodule structure on $K^\op\otimes H$ is
$$\begin{array}{cccc}
(\id\otimes\pi)\circ\Delta_{K^{\op}\otimes H}:
& K^\op\otimes H & \rightarrow & (K^\op\otimes H)\otimes K^\op, \;\; \\
& k\otimes h & \mapsto &
\sum (k_{(1)}\otimes h_{(1)})\otimes k_{(2)}\sigma_r(h_{(2)}).
\end{array}$$

\item[(5)]
Note that $\iota$ defined in (\ref{def: iota}) and $\pi$ defined in (\ref{def: pi}) both preserve the units and counits of the Hopf algebras.
Then it is easy to see that $\zeta$ and $\gamma$ are biunitary.
\item[(6)]
The final condition $(\iota\circ\zeta)\ast(\gamma\circ\pi)=\id_{K^\op\otimes H}$ can be verified directly.
\end{itemize}
\end{proof}

Finally, the main result of this subsection can be introduced.

\begin{theorem}\label{prop: lpd qd}
The left partial dual $K^{\ast\cop}\# H$ of $K^\op\otimes H$ determined by the partially admissible mapping system $(\zeta, \gamma^\ast)$ in Lemma \ref{lem:PAMSleft} is the quantum double $K^{\ast\cop}\bowtie_\sigma H$.
\end{theorem}

\begin{proof}
Consider the partially admissible mapping system $(\zeta, \gamma^\ast)$ in (\ref{eqn:admissiblemapsys}), and
recall that the right $K^\op\otimes H$-module structure (\ref{eqn:r Kop tensor H ms})
\begin{equation*}
\btl: K^\op\otimes(K^\op\otimes H)\rightarrow K^\op,  \;\;\;k'\otimes(k\otimes h)\mapsto kk'\sigma_r(h),
\end{equation*}
of $K^\op$ will induce the right $K^{\ast\cop}\otimes H^\ast$-comodule structure of $K^{\ast\cop}$, which is as follows:
\begin{equation}\label{eqn: rightcoidealsubalg}
\begin{array}{ccc}
K^{\ast\cop} &\rightarrow& K^{\ast\cop}\otimes (K^{\ast\cop}\otimes H^\ast)  \\
k^\ast &\mapsto& \sum k^\ast_{(2)}\otimes(k^\ast_{(1)}\otimes\sigma_l(k^\ast_{(3)})).
\end{array}
\end{equation}
This is because the equations
\begin{eqnarray*}
\big\langle \sum k^\ast_{(2)}\otimes(k^\ast_{(1)}\otimes\sigma_l(k^\ast_{(3)})),
k'\otimes(k\otimes h)\big\rangle
&\overset{(\ref{eqn:indentification})}{=}&
\sum\langle k^\ast_{(2)},k'\rangle\langle k^\ast_{(1)},k\rangle\langle\sigma_l(k^\ast_{(3)}),h\rangle  \\
&\overset{\text{Notation}\;\ref{not:sigma lr}}=&
\langle k^\ast,kk'\sigma_r(h)\rangle
\end{eqnarray*}
hold for all $k^\ast\in K^\ast$, $k,k'\in K$ and $h\in H$.

Using the notation (\ref{hit action}), we write
\begin{eqnarray}\label{eqn: hit}
&&(k\otimes h)\rightharpoonup k^\ast
\overset{(\ref{eqn: rightcoidealsubalg})}{=}
\sum k^\ast_{(2)}\big\langle k^\ast_{(1)}\otimes\sigma_l(k^\ast_{(3)}),
  k\otimes h\big\rangle  \nonumber
\overset{(\ref{eqn:indentification})}{=}
\sum k^\ast_{(2)}\langle k^\ast_{(1)},k\rangle\langle\sigma_l(k^\ast_{(3)}),
   h\rangle  \\
&& \;\;\;\;\;\;\;\;\;\;\;\;\;\;\;\;\;\;\;\;\;\;\;\;\;\;\;\;\;\;\;\;\;\;\;\;
\;\;\;\;\;\;\;\;\;\;\;\;\;\;\;\;\;\;\;\;\;\;\;\;\;\;\;\;
(\forall k\in K^\op,\;\forall h\in H,\;\forall k^\ast\in K^{\ast\cop}).
\end{eqnarray}

Now we can proceed to formulate the algebra structure of the left partial dual $K^{\ast\cop}\# H$.
According to \cite[Equation (3.1)]{Li23}, the multiplication is given by: For all $k^\ast, k'^\ast\in K^\ast$ and $h, h'\in H$,
\begin{eqnarray*}
(k^\ast\# h)(k'^\ast\# h')
&=&
\sum k^\ast\left([\sigma_r(S^{-1}(h_{(3)}))\otimes h_{(1)}]\rightharpoonup k'^\ast\right)\# h_{(2)}h' \\
&\overset{(\ref{eqn: hit})}{=}&
\sum k^\ast k'^\ast_{(2)}\langle k'^\ast_{(1)}, \sigma_r(S^{-1}(h_{(3)}))\rangle\langle \sigma_l(k'^\ast_{(3)}), h_{(1)}\rangle\# h_{(2)}h' \\
&\overset{(\ref{eqn: sigma lr})}{=}&
\sum k^\ast k'^\ast_{(2)}\sigma(k'^\ast_{(1)}, S^{-1}(h_{(3)}))\sigma(k'^\ast_{(3)}, h_{(1)})\# h_{(2)}h' \\
&\overset{(\ref{eqn: sigmabar})}=&
\sum k^\ast k'^\ast_{(2)}\overline{\sigma}(k'^\ast_{(1)}, h_{(3)}) \sigma(k'^\ast_{(3)}, h_{(1)})\# h_{(2)}h'  \\
&=&
\sum \sigma(k'^\ast_{(3)}, h_{(1)})k^\ast k'^\ast_{(2)} \# h_{(2)}h'\overline{\sigma}(k'^\ast_{(1)}, h_{(3)}),
\end{eqnarray*}
which coincides with products (\ref{eqn: multiplication}) in the quantum double $K^{\ast\cop}\otimes H^\ast$.
Besides, the unit element is $\varepsilon\#1$.

On the other hand, note that
$\zeta$ defined in (\ref{zeta}) and $\gamma$ defined in (\ref{gamma}) are clearly both Hopf algebra maps. It follows from Lemma \ref{lem: lpd is Hopf alg} that $K^{\ast\cop}\# H$ is a Hopf algebra, and its coalgebra structure is the tensor product $K^{\ast\cop}\otimes H^\ast$.

Finally, we conclude that $K^{\ast\cop}\# H$ and $K^{\ast\cop}\bowtie_\sigma H$ are the same Hopf algebras.
\end{proof}

In particular, we
can obtain the following observation on the Drinfeld double.

\begin{corollary}
The Drinfeld double $D(H)$ of $H$ is a left partially dualized Hopf algebra of $H^\op\otimes H$.
\end{corollary}

\begin{remark}
This corollary
can be regarded as a Hopf algebraic version of \cite[Proposition 2.5]{Ost03}.
\end{remark}


\subsection{Tensor equivalences to the category of relative Yetter-Drinfeld modules}

To begin with, let $H$ be a finite-dimensional Hopf algebra, and let $B$ be a left $H$-comodule algebra and $C$ a right $H$-module coalgebra. As usual,
we will use the Sweedler $\Sigma$-notation for $x^\ast\in C^\ast$ and $v\in V$,
$$
x^\ast\mapsto\sum x^\ast_{(1)}\otimes x^\ast_{(2)}\in C^\ast\otimes H^\ast\;\;\;\text{and}\;\;\;v\mapsto\sum v_{\langle0\rangle}\otimes v_{\langle1\rangle}\in V\otimes B^\ast
$$
to represent the right $H^\ast$-comodule structure of $C^\ast$ and the right $B^\ast$-comodule structure of $V$, respectively.

Consider the $\k$-linear abelian category ${}_{C^\ast}\M^{B^\ast}$, which consists of finite-dimensional vector spaces $V$ with both a left $C^\ast$-module and a right $B^\ast$-comodule structure, satisfying the compatibility condition
\begin{equation}\label{eqn:lmrccom1}
\sum(x^\ast v)_{\langle0\rangle}\otimes(x^\ast v)_{\langle1\rangle}
=\sum x^\ast_{(1)}v_{\langle0\rangle}\otimes(x^\ast_{(2)}\btr
v_{\langle1\rangle})\;\;\;\;\;\;\;\;(\forall x^\ast\in C^\ast,\;\forall v\in V),
\end{equation}
where $\btr$ denotes the left $H^\ast$-action on $B^\ast$ induced by the left $H$-comodule structure on $B$, namely:
\begin{equation}\label{eqn: btr}
\langle h^\ast\btr b^\ast,b\rangle=\sum\la h^\ast,b_{(1)}\ra\la b^\ast,b_{(2)}\ra
\end{equation}
holds for all $h^\ast\in H$, $b^\ast\in B^\ast$ and $b\in B$.

We remark that the ${}_{C^\ast}\M^{B^\ast}$ is referred
to as the category of Doi-Hopf modules in \cite{CMZ97, CMIZ99},
and the first canonical equivalence (in fact, isomorphism) is due to \cite[Remark (1.3)(b)]{Doi92}.
\begin{lemma}(\cite[Remark (1.3)(b)]{Doi92})\label{lem: Rep cong}
Let $H$ be a finite-dimensional Hopf algebra, and let $B$ be a left $H$-comodule algebra and $C$ a right $H$-module coalgebra.
Then
$${}^{}_{C^\ast}\M^{B^\ast}\cong\Rep\left(C^\ast\# B\right)$$
as $\k$-linear abelian categories, which sends each $V\in{}^{}_{C^\ast}\M^{B^\ast}$ to the left $C^\ast\# B$-module $V$ with structure defined via
\begin{equation}\label{eqn: l pd mod}
(x^\ast\# b)\cdot v=
\sum x^\ast v_{\langle0\rangle}\langle v_{\langle1\rangle}, b\rangle
\;\;\;\;\;\;\;\;(\forall x^\ast\in C^\ast,\;\forall b\in B,\;\forall v\in V).
\end{equation}
\end{lemma}

With the help of this lemma, we can establish a tensor isomorphism
for $\Rep(K^{\ast\cop}\bowtie_\sigma H)$ in the following proposition.

Recall in Subsection \ref{subsection 2.2} that $\big({}_{K^\ast}\YD^{H^\ast}\big)^\rev$ is the category of (left-right) Yetter-Drinfeld modules with Hopf pairing
$\sigma'$, and it has the tensor product bifunctor defined according to (\ref{eqn: lK* *YD*}) and (\ref{eqn: rH* *YD*}).

\begin{proposition}\label{prop: *YD* cong rep}
Suppose that
\begin{equation}
\begin{array}{ccc}
\xymatrix{
H\ar@<.5ex>[r]^{\iota\;\;\;\;\;\;}
& K^\op\otimes H
 \ar@<.5ex>@{-->}[l]^{\zeta\;\;\;\;\;\;} \ar@<.5ex>[r]^{\;\;\;\;\;\;\pi}
& K^\op \ar@<.5ex>@{-->}[l]^{\;\;\;\;\;\;\gamma}  }
&\;\;\text{and}\;\;&
\xymatrix{
K^{\ast\cop}\ar@<.5ex>[r]^{\pi^\ast\;\;\;\;\;\;}
& K^{\ast\cop}\otimes H^\ast
 \ar@<.5ex>@{-->}[l]^{\gamma^\ast\;\;\;\;\;\;}
 \ar@<.5ex>[r]^{\;\;\;\;\;\;\;\;\iota^\ast}
& H^\ast \ar@<.5ex>@{-->}[l]^{\;\;\;\;\;\;\;\;\zeta^\ast}  }
\end{array}
\end{equation}
is the partially admissible mapping system $(\zeta, \gamma^\ast)$ defined in Lemmas \ref{lem:PAMSleft0} and \ref{lem:PAMSleft}. Then
\begin{equation}\label{eqn: tensor iso}
\Theta:\big({}_{K^\ast}\YD^{H^\ast}\big)^\rev\cong\Rep(K^{\ast\cop}\bowtie_\sigma H)
\end{equation}
as tensor categories, which sends each $V\in\big({}_{K^\ast}\YD^{H^\ast}\big)^\rev$ to the left $K^{\ast\cop}\bowtie_\sigma H$-module $\Theta(V)$ with underlying vector space $V$ and structure defined via
\begin{equation}\label{eqn: l pd mod}
(k^\ast\bowtie h)\cdot v=
\sum k^\ast v_{\langle0\rangle}\langle v_{\langle1\rangle}, h\rangle
\;\;\;\;\;\;\;\;(\forall k^\ast\in K^{\ast\cop},\;\forall h\in H,\;\forall v\in V),
\end{equation}
where $v\mapsto\sum v_{\langle0\rangle}\otimes v_{\langle1\rangle}$ denotes the right $H^\ast$-comodule structure on $V$.
\end{proposition}

\begin{proof}
We start by recalling in Theorem \ref{prop: lpd qd} that $K^{\ast\cop}\bowtie_\sigma H$ is the left partial dualized Hopf algebra $K^{\ast\cop}\#H$ determined by the partially admissible mapping system in (\ref{eqn:admissiblemapsys}). Then it follows by
Lemma \ref{lem: Rep cong} that there is an isomorphism
${}^{}_{K^{\ast\cop}}\M^{H^\ast}\cong\Rep(K^{\ast\cop}\bowtie_\sigma H)$ of $\k$-linear abelian categories.

Now we claim that the category ${}^{}_{K^{\ast\cop}}\M^{H^\ast}$ coincides exactly with $\big({}_{K^\ast}\YD^{H^\ast}\big)^\rev$,
as the compatibility condition (\ref{eqn:lmrccom1}) satisfied
for objects in the former category is in fact identical to those in $\big({}_{K^\ast}\YD^{H^\ast}\big)^\rev$.

In order to show this, note that the left $K^{\ast\cop}\otimes H^\ast$-module structure $\btr$ on $H^\ast$ is induced as
\begin{equation}\label{eqn: btr1}
(k^\ast\otimes h^\ast)\btr h'^\ast=\sum h^\ast h'^\ast S^{-1}(\sigma_l(k^\ast))
\end{equation}
for any $k^\ast\in K^{\ast\cop}$ and $h^\ast, h'^\ast\in H^\ast$,
since the equations
\begin{eqnarray*}
\langle (k^\ast\otimes h^\ast)\btr h'^\ast, h\rangle
&\overset{(\ref{eqn: btr})}=&
\sum\big\la k^\ast\otimes h^\ast,\sigma_r(S^{-1}(h_{(3)}))\otimes h_{(1)}\big\ra\big\la h'^\ast, h_{(2)}\big\ra \\
&=&
\sum \big\la k^\ast,\sigma_r(S^{-1}(h_{(2)}))\big\ra\la h^\ast h'^\ast, h_{(1)}\ra \\
&=&
\sum \big\la h^\ast h'^\ast S^{-1}(\sigma_l(k^\ast)), h\big\ra
\end{eqnarray*}
hold for all $h\in H$.

Moreover, suppose $V \in {}^{}_{K^{\ast\cop}}\M^{H^\ast}$, and the compatibility condition (\ref{eqn:lmrccom1}) implies that
\begin{eqnarray*}\label{eqn: lK*modrH*comod of V}
\sum (k^\ast v)_{\la0\ra}\otimes (k^\ast v)_{\la1\ra}
&\overset{(\ref{eqn: rightcoidealsubalg}),\;(\ref{eqn:lmrccom1})}=&
\sum k^\ast_{(2)}v_{\la0\ra}\otimes\big((k^\ast_{(1)}
\otimes\sigma_l(k^\ast_{(3)}))\btr v_{\la1\ra}\big)\\
&\overset{(\ref{eqn: btr1})}=&
\sum k^\ast_{(2)}v_{\la0\ra}\otimes \sigma_l(k^\ast_{(3)})v_{\la1\ra}S^{-1}(\sigma_l(k^\ast_{(1)}))  \\
&=&
\sum k^\ast_{(2)}v_{\la0\ra}\otimes \sigma'_r(k^\ast_{(3)})v_{\la1\ra}S^{-1}(\sigma'_r(k^\ast_{(1)}))
\end{eqnarray*}
for all $k^\ast\in K^{\ast\cop}$ and $v\in V$.
However, it is straightforward to verify that this equality agrees with the defining condition (\ref{eqn:lmrccom1}) for $V$ becoming an object in $\big({}_{K^\ast}\YD^{H^\ast}\big)^\rev$.
In conclusion, the category ${}^{}_{K^{\ast\cop}}\M^{H^\ast}$ is the same as $\big({}_{K^\ast}\YD^{H^\ast}\big)^\rev$, and consequently $\Theta$ (\ref{eqn: tensor iso}) is an isomorphism of $\k$-linear abelian categories.

Let us proceed to show that $\Theta$ is a tensor functor.
It suffices to check that for all $V, W\in \big({}_{K^\ast}\YD^{H^\ast}\big)^\rev$, the identity map
\begin{equation}\label{eqn: id mor}
\id_{V\otimes W}: \Theta(V)\otimes \Theta(W) \cong \Theta(V\otimes W),\;\;
v\otimes w \mapsto v\otimes w
\end{equation}
on $V\otimes W$ is a morphism in $\Rep(K^{\ast\cop}\bowtie_\sigma H)$. Our goal is to show that the $K^{\ast\cop}\bowtie_\sigma H$-module structures on $\Theta(V)\otimes \Theta(W)$ and $\Theta(V\otimes W)$ coincide.

Indeed, recall that $\Theta(V)$ and $\Theta(W)$ admit left $K^{\ast\cop}\bowtie_\sigma H$-module structures as in (\ref{eqn: l pd mod}). Then the $K^{\ast\cop}\bowtie_\sigma H$-action on their tensor product $\Theta(V)\otimes \Theta(W)$
is diagonal, namely: For any $k^\ast\in K^{\ast\cop}$, $h\in H$ and $v\in V$, $w\in W$,
\begin{eqnarray}\label{eqn: lpd mod of tp}
(k^\ast\bowtie h)\cdot(v\otimes w)
&=&
\sum\big((k^\ast_{(2)}\bowtie h_{(1)})\cdot v\big)\otimes \big((k^\ast_{(1)}\bowtie h_{(2)})\cdot w\big)
\nonumber  \\
&\overset{(\ref{eqn: l pd mod})}=&
\sum k^\ast_{(2)}v_{\la0\ra}\la v_{\la1\ra}, h_{(1)}\ra\otimes k^\ast_{(1)}w_{\la0\ra}\la w_{\la1\ra}, h_{(2)}\ra.
\end{eqnarray}

On the other hand, for objects $V,W\in\big({}_{K^\ast}\YD^{H^\ast}\big)^\rev$, we know in (\ref{eqn: lK* *YD*}) and (\ref{eqn: rH* *YD*}) that $V\otimes W$ is also an object of $\big({}_{K^\ast}\YD^{H^\ast}\big)^\rev$, where
\begin{equation}\label{eqn: mod stru of tp}
k^\ast\cdot(v\otimes w)=\sum k^\ast_{(2)}v\otimes k^\ast_{(1)}w
\end{equation}
and
\begin{equation}\label{eqn: comod stru of tp}
\sum(v\otimes w)_{\la0\ra}\otimes (v\otimes w)_{\la1\ra}=\sum(v_{\la0\ra}\otimes w_{\la0\ra})\otimes v_{\la1\ra}w_{\la1\ra}
\end{equation}
hold for all $k^\ast\in K^\ast$, $v\in V$ and $w\in W$.
Furthermore, $\Theta(V\otimes W)$ becomes an object in $\Rep(K^{\ast\cop}\bowtie_\sigma H)$ with the action determined by (\ref{eqn: l pd mod}) as
\begin{eqnarray}\label{eqn: lpd mod of tp2}
(k^\ast\bowtie h)\cdot(v\otimes w)
&\overset{(\ref{eqn: comod stru of tp})}=&
\sum k^\ast\cdot(v_{\la0\ra}\otimes w_{\la0\ra})\big\la v_{\la1\ra}w_{\la1\ra}, h\big\ra \nonumber \\
&\overset{(\ref{eqn: mod stru of tp})}=&
\sum k^\ast_{(2)}v_{\la0\ra}\la v_{\la1\ra}, h_{(1)}\ra\otimes k^\ast_{(1)}w_{\la0\ra}\la w_{\la1\ra}, h_{(2)}\ra
\end{eqnarray}
for any $k^\ast\in K^{\ast\cop}$, $h\in H$ and $v\in V$, $w\in W$.
Since (\ref{eqn: lpd mod of tp2}) is equal to (\ref{eqn: lpd mod of tp}),
we can conclude that the identity morphism (\ref{eqn: id mor}) is the monoidal structure of $\Theta$, which is consequently a tensor isomorphism.
\end{proof}

Based on Propositions \ref{prop: YD cong *YD*} and \ref{prop: *YD* cong rep}, we can generalize Lemma \ref{lem: HYDH cong D(H)} as follows.

\begin{corollary}\label{cor: YD iso repqd}
There is an isomorphism of tensor categories
\begin{equation}\label{eqn: YD iso repqd}
{}_H\YD^K\cong\Rep(K^{\ast\cop}\bowtie_\sigma H).
\end{equation}
Specifically, for each object $V\in{}_H\YD^K$, the left $K^{\ast\cop}\bowtie_\sigma H$-action on $V$ is defined by
\begin{equation}\label{eqn: YD iso repqd action}
(k^\ast\bowtie h)\cdot v=\sum (h\cdot v)_{\la0\ra} \la k^\ast,(h\cdot v)_{\la1\ra}\ra
\end{equation}
for all $k^\ast\in K^\ast$, $h\in H$ and $v\in V$.
\end{corollary}

Similarly, we also have a tensor isomorphism
${}_{K^\ast}\YD^{H^\ast}\cong\Rep(H^\cop\bowtie_{\sigma'} K^\ast)$ as an application of Corollary \ref{cor: YD iso repqd} to the Hopf pairing $\sigma':H\otimes K^\ast\rightarrow\k,$\; $h\otimes k^\ast\mapsto\sigma(k^\ast,h)$.
Note that $\Rep(H^\cop\bowtie_{\sigma'} K^\ast)^\rev$ reconstructs the coopposite Hopf algebra $(H^\cop\bowtie_{\sigma'} K^\ast)^\mathrm{cop}$.

\begin{corollary}
The Hopf algebras
$(H^\cop\bowtie_{\sigma'} K^\ast)^\mathrm{cop}$ and $K^{\ast\cop}\bowtie_\sigma H$ are gauge equivalent.
\end{corollary}

\begin{proof}
It follows from
Proposition
\ref{prop: YD cong *YD*} and Corollary \ref{cor: YD iso repqd} that
$$\Rep(H^\cop\bowtie_{\sigma'} K^\ast)^\rev\cong \big({}_{K^\ast}\YD^{H^\ast}\big)^\rev\cong{}_H\YD^K\cong\Rep(K^{\ast\cop}\bowtie_\sigma H)$$
as finite tensor categories. The claim holds as a consequence of \cite[Theorem 2.2]{NS08}.
\end{proof}

\subsection{Dual tensor categories from the reconstruction of the quantum double}

Next, by applying Lemma \ref{lem: reconstructionthm}, we obtain the other canonical (tensor) equivalence for the category $\Rep(K^{\ast\cop}\bowtie_\sigma H)$, which is
formalized as the following proposition.
The notion of the \textit{cotensor product} $-\square_C-$ over a coalgebra $C$ would be used, and one might refer to \cite[Section 0]{Tak77} for the definition and basic properties.

\begin{proposition}\label{prop: te lDoi}
Let ${}^{}_{K^{\ast\cop}}\M^{K^{\ast\cop}\otimes H^\ast}_{K^{\ast\cop}}$ denote the finite tensor category of finite-dimensional $K^{\ast\cop}$-$K^{\ast\cop}$-bimodules $M$ equipped with right $K^{\ast\cop}\otimes H^\ast$-comodule structure $m\mapsto\sum m_{(0)}\otimes m_{(1)}$
satisfying that
\begin{equation}\label{eqn: rcomod of lk*mod}
\sum (k^\ast\cdot m)_{(0)}\otimes(k^\ast\cdot m)_{(1)}
=\sum k^\ast_{(2)}\cdot m_{(0)}
 \otimes (k^\ast_{(1)}\otimes\sigma_l(k^\ast_{(3)}))m_{(1)},
\end{equation}
\begin{equation}\label{eqn: rcomod of rk*mod}
\sum (m\cdot k^\ast)_{(0)}\otimes(m\cdot k^\ast)_{(1)}
=\sum m_{(0)}\cdot k^\ast_{(2)}
 \otimes m_{(1)}(k^\ast_{(1)}\otimes\sigma_l(k^\ast_{(3)}))
\end{equation}
for all $k^\ast\in K^{\ast\cop}$ and $m\in M$.
Then there is a tensor equivalence
\begin{equation}\label{eqn:catequ1}
{}^{}_{K^{\ast\cop}}\M^{K^{\ast\cop}\otimes H^\ast}_{K^{\ast\cop}}
\approx\Rep(K^{\ast\cop}\bowtie_\sigma H)
\end{equation}
given by the functors
\begin{equation}\label{eqn:cat func1}
M\mapsto M/M(K^{\ast\cop})^+
\;\;\;\;\;\text{and}\;\;\;\;\;
V\square_{H^\ast}(K^{\ast\cop}\otimes H^\ast)\mapsfrom V.
\end{equation}
\end{proposition}

\begin{proof}
Note that the right $K^\op\otimes H$-module coalgebra map $\pi$ (\ref{def: pi}) defines the category ${}^{}_{K^{\ast\cop}}\M^{K^{\ast\cop}\otimes H^\ast}_{K^{\ast\cop}}$ of relative Doi-Hopf modules as introduced before Lemma \ref{lem: reconstructionthm}. Indeed, the compatibility conditions (\ref{eqn: rcomod of lx*mod}) and (\ref{eqn: rcomod of rx*mod}) for each object $M$ will respectively become (\ref{eqn: rcomod of lk*mod}) and (\ref{eqn: rcomod of rk*mod}) in this situation.

Moreover, we know by Theorem \ref{prop: lpd qd} that the quantum double $K^{\ast\cop}\bowtie_\sigma H$ is a left partial dualized Hopf algebra of $K^\op\otimes H$, and our desired equivalences (\ref{eqn:cat func1}) are obtained by Lemma \ref{lem: reconstructionthm} and the functors $\Phi$ and $\Psi$ defined in \cite[Section 1]{Tak79}.
\end{proof}

\begin{remark}
It is clear that $\iota$ (\ref{def: iota}) induces
$\iota^\ast: K^{\ast\cop}\otimes H^\ast\rightarrow H^\ast,\;
k^\ast\otimes h^\ast\mapsto h^\ast S^{-1}(\sigma_l(k^\ast))$, and
we note in the proof of \cite[Lemma 4.9]{Li23} that the left $H^\ast$-comodule structure of $K^{\ast\cop}\otimes H^\ast$ is considered as $(\iota^\ast\otimes\id)\circ\Delta:$
$$
K^{\ast\cop}\otimes H^\ast\rightarrow H^\ast\otimes(K^{\ast\cop}\otimes H^\ast),\;\;k^\ast\otimes h^\ast\mapsto\sum h^\ast_{(1)}S^{-1}(\sigma_l(k^\ast_{(2)}))\otimes(k^\ast_{(1)}\otimes h^\ast_{(2)}).
$$
Therefore, for each right $H^\ast$-comodule $V$, the cotensor product $V\square_{H^\ast}(K^{\ast\cop}\otimes H^\ast)$ consists of elements $\sum_{i}v_i\otimes(k_i^\ast\otimes h_i^\ast)$ in $V\otimes(K^{\ast\cop}\otimes H^\ast)$ satisfying
\begin{equation}\label{eqn: cotp}
\sum_{i}v_{i\langle0\rangle}\otimes
v_{i\langle1\rangle}\otimes(k_i^\ast\otimes h_i^\ast)=\sum_{i}v_i\otimes {h_i^\ast}_{(1)}S^{-1}(\sigma_l({k_i^\ast}_{(2)}))\otimes
({k_i^\ast}_{(1)}\otimes {h_i^\ast}_{(2)}).
\end{equation}
\end{remark}

In fact, the expression of $V\square_{H^\ast}(K^{\ast\cop}\otimes H^\ast)$ can be simplified. To this end, we show that it is linearly isomorphic to $V\otimes K^\ast$, which is then regarded as an object in ${}^{}_{K^{\ast\cop}}\M^{K^{\ast\cop}\otimes H^\ast}_{K^{\ast\cop}}$.

\begin{lemma}\label{lem:leftcoiso}
For each $V\in\Rep(K^{\ast\cop}\bowtie_\sigma H)$, there is a $\k$-linear isomorphism
\begin{equation}\label{phi}
\phi:
V\square_{H^\ast}(K^{\ast\cop}\otimes H^\ast)\cong V\otimes K^\ast,\;\;\sum_iv_i\otimes(k_i^\ast\otimes h_i^\ast)\mapsto \sum_iv_i\otimes k_i^\ast\la h_i^\ast,1\ra,
\end{equation}
which makes $V\otimes K^\ast\in{}^{}_{K^{\ast\cop}}\M^{K^{\ast\cop}\otimes H^\ast}_{K^{\ast\cop}}$ with structures:
\begin{itemize}
\item[\rm{(1)}]
The left $K^{\ast\cop}$-action is diagonal and the right $K^{\ast\cop}$-action is defined through the second tensorand $K^\ast$, respectively given by
\begin{equation}\label{mods of VotimesK*}
l^\ast\cdot(v\otimes k^\ast)=\sum l^\ast_{(2)}v\otimes l^\ast_{(1)}k^\ast\;\;\;\;\text{and}\;\;\;\;(v\otimes k^\ast)\cdot l^\ast=\sum v\otimes k^\ast l^\ast
\end{equation}
for any $l^\ast\in K^{\ast\cop}$, $v\in V$ and $k^\ast\in K^\ast$.

\item[\rm{(2)}]
The right $K^{\ast\cop}\otimes H^\ast$-coaction on $V\otimes K^\ast$ is defined as
\begin{equation}\label{comods of VotimesK*}
v\otimes k^\ast\mapsto\sum (v_{\langle0\rangle}\otimes k^\ast_{(2)})\otimes (k^\ast_{(1)}\otimes v_{\langle1\rangle}\sigma_l(k^\ast_{(3)})),
\end{equation}
where $\sum v_{\langle0\rangle}\otimes v_{\langle1\rangle}\in V\otimes H^\ast$ satisfies Equation (\ref{eqn: l pd mod}).
\end{itemize}
In other words, $\phi$ is regarded as an isomorphism in ${}^{}_{K^{\ast\cop}}\M^{K^{\ast\cop}\otimes H^\ast}_{K^{\ast\cop}}$.
\end{lemma}

\begin{proof}
We start by defining a linear map
\begin{equation}\label{psi}
\psi:V\otimes K^\ast\rightarrow V\square_{H^\ast}(K^{\ast\cop}\otimes H^\ast),\;\;
v\otimes k^\ast\mapsto\sum v_{\langle0\rangle}\otimes(k^\ast_{(1)}
\otimes v_{\langle1\rangle}\sigma_l(k^\ast_{(2)})),
\end{equation}
which is well-defined because the image satisfies the condition (\ref{eqn: cotp}), namely:
$$\sum v_{\langle0\rangle}\otimes v_{\langle1\rangle}\otimes (k^\ast_{(1)}\otimes v_{\langle2\rangle}\sigma_l(k^\ast_{(2)}))
=
\sum v_{\langle0\rangle}\otimes v_{\langle1\rangle}\sigma_l(k^\ast_{(3)})S^{-1}
(\sigma_l(k^\ast_{(2)}))\otimes \big(k^\ast_{(1)}\otimes v_{\langle2\rangle}\sigma_l(k^\ast_{(4)})\big).
$$
Furthermore,
using Equation (\ref{eqn: cotp})
we can directly find that
$\psi$ is the inverse of $\phi$, and hence $\phi$ is a linear isomorphism.

However, we know according to \cite[Lemma 4.9]{Li23} that for each $V\in\Rep(K^{\ast\cop}\bowtie_\sigma H)$, the left $K^{\ast\cop}$-action on $V\square_{H^\ast}(K^{\ast\cop}\otimes H^\ast)$ should be diagonal via the right $K^{\ast\cop}\otimes H^\ast$-comodule structure
(\ref{eqn: rightcoidealsubalg}):
\begin{equation}\label{eqn: lK*cop of cotp}
l^\ast\cdot\big[\sum_{i}v_i\otimes(k^\ast_i\otimes h^\ast_i)\big]=\sum l^\ast_{(2)}v_i\otimes(l^\ast_{(1)}
k^\ast_i\otimes\sigma_l(l^\ast_{(3)})h^\ast_i),
\end{equation}
for any $l^\ast\in K^{\ast\cop}$. The right $K^{\ast\cop}$-action and $K^{\ast\cop}\otimes H^\ast$-coaction on $V\square_{H^\ast}(K^{\ast\cop}\otimes H^\ast)$ are given through the second (co)tensorand $K^{\ast\cop}\otimes H^\ast$, respectively:
\begin{equation}\label{eqn: rK*cop of cotp}
\big[\sum_{i}v_i\otimes(k^\ast_i\otimes h^\ast_i)\big]\cdot l^\ast=\sum v_i\otimes k^\ast_il^\ast_{(2)}\otimes h^\ast_i\sigma_l(l^\ast_{(1)})
\end{equation}
via $\pi^\ast:K^{\ast\cop}\rightarrow K^{\ast\cop}\otimes H^\ast,\;
l^\ast\mapsto \sum l^\ast_{(2)}\otimes \sigma_l(l^\ast_{(1)})$
induced by $\pi$ (\ref{def: pi}), as well as
\begin{equation}\label{eqn: rK*copcomod of cotp}
\sum_{i}v_i\otimes(k^\ast_i\otimes h^\ast_i)\mapsto\sum_{i}\big[v_i\otimes({k^\ast_i}_{(2)}\otimes
{h^\ast_i}_{(1)})\big]\otimes({k^\ast_i}_{(1)}\otimes{h^\ast_i}_{(2)}).
\end{equation}

Finally,
by direct computations using Equation (\ref{eqn: cotp}), one can show that $\phi$ transfers the above actions and coaction into (\ref{mods of VotimesK*}) and (\ref{comods of VotimesK*}).
As a result, the structures (\ref{mods of VotimesK*}) and (\ref{comods of VotimesK*}) make $V\otimes K^\ast$
into an object
in ${}^{}_{K^{\ast\cop}}\M^{K^{\ast\cop}\otimes H^\ast}_{K^{\ast\cop}}$ which is isomorphic to $V\square_{H^\ast}(K^{\ast\cop}\otimes H^\ast)$, and the proof is completed.
\end{proof}

With the help of this lemma, the form of the equivalences in Proposition \ref{prop: te lDoi} can be simplified as follows.
\begin{corollary}\label{cor: te lDoi}
There are mutually quasi-inverse equivalences
$$\begin{array}{cccc}
\Phi :& {}^{}_{K^{\ast\cop}}\M^{K^{\ast\cop}\otimes H^\ast}_{K^{\ast\cop}} &\rightarrow& \Rep(K^{\ast\cop}\bowtie_\sigma H),  \\
& M &\mapsto& M/M(K^{\ast\cop})^+
\end{array}$$
and
$$\begin{array}{cccc}
\Psi :& \Rep(K^{\ast\cop}\bowtie_\sigma H) &\rightarrow& {}^{}_{K^{\ast\cop}}\M^{K^{\ast\cop}\otimes H^\ast}_{K^{\ast\cop}},  \\
& V &\mapsto& V\otimes K^\ast
\end{array}$$
of finite tensor categories.
\end{corollary}

\begin{proof}
It suffices to note that $\Psi$ is naturally isomorphic to the functor $V\mapsto V\square_{H^\ast}(K^{\ast\cop}\otimes H^\ast)$ introduced in Proposition \ref{prop: te lDoi}, but this is evident.
\end{proof}

\subsection{A particular case and the relative center observation}

Recall
from Corollary \ref{cor: YD iso repqd} that an isomorphism
\begin{equation}\label{eqn: YD iso repqd2}
{}_H\YD^K\cong\Rep(K^{\ast\cop}\bowtie_\sigma H).
\end{equation}
of finite tensor categories has been provided.
In this subsection, we focus on the situation when the Hopf algebra map $\sigma_r:H\rightarrow K$ is surjective, and explain that ${}_H\YD^K$ (as well as $\Rep(K^{\ast\cop}\bowtie_\sigma H)$) is an example of the relative center of the finite tensor category $\Rep(H)$ (with respect to the tensor subcategory $\Rep(K)$ induced via $\sigma_r)$. In addition, some related results in the literature would be mentioned.

The notion of relative centers
can be found for example in \cite[Section 2B]{GNN09}, but we consider a ``left center version'' in the following definition.

\begin{definition}\label{def:relativecenter}
Let $\C$ and $\D$ be finite tensor categories,
with $\D$ a tensor subcategory of $\C$.
Their tensor product bifunctors are both denoted by $\otimes$, and the associativity constraints are abbreviated for convenience.
The category $\Z_\D(\C)$ is defined as follows:
\begin{itemize}
\item
The objects are pairs $(X,c)$, where $X\in\C$ and $c$ is a family of natural isomorphisms
$$c_Y:X\otimes Y\xrightarrow{\sim}Y\otimes X\;\;\;\;\;\;\;\;(Y\in\D)$$
satisfying
$$(\id_{Y}\otimes c_{Y'})\circ (\id_{Y}\otimes c_{Y'})
=c_{Y\otimes Y'}$$
for all $Y,Y'\in\D$;
\item
A morphism from $(X,c)$ to $(X',c')$ is a morphism $f\in\Hom_\C(X,X')$ such that
$$(\id_{Y}\otimes f)\circ c_Y=c'_Y\circ(f\otimes \id_{Y})$$
for each $Y\in\D$.
\end{itemize}
It is a finite tensor category with tensor product bifunctor $\otimes$ given as follows: For any $(X,c),(X',c')\in\Z_\D(\C)$, define
$$(X,c)\otimes(X',c'):=(X\otimes X',\tilde{c})$$
where
$\tilde{c}_Y=(c_Y\otimes\id_{X'})\circ(\id_X\otimes c'_Y)$ for all $Y\in\D$.
\end{definition}

\begin{remark}\label{rmk:relativecenter}
\begin{itemize}
\item[\rm{(1)}]
Clearly, when $\D=\C$, the category $\Z_\D(\C)$ becomes the left center of $\C$ (e.g. \cite{JS91});
\item[\rm{(2)}]
The ``right center'' version of $\Z_\D(\C)$
would be the category $\overline{\Z}_\D(\C)$ of
objects of the form
$(X,c^{-1})$ for $(X,c)\in\Z_\D(\C)$, which coincides with the notion introduced in \cite[Section 2B]{GNN09};
\item[\rm{(3)}]
There is another kind of generalization of $\Z(\C)$ introduced in \cite{Lau20}, which may be applied to our settings if $\D$ is braided.
\end{itemize}
\end{remark}

Under the assumptions in Definition \ref{def:relativecenter},
one may establish a tensor equivalence
\begin{equation}\label{eqn:relativecenter-equiv}
\Z_\D(\C)\approx(\D^\mathrm{rev}\boxtimes\C)_\C^\ast
\end{equation}
as a consequence analogous to \cite[Remark 2.4]{GNN09}.
Here, we do not explain the construction of the functor (\ref{eqn:relativecenter-equiv}) in detail, but remark that the structure of $\C$ as a left module category over the Deligne's tensor product $\D^\mathrm{rev}\boxtimes\C$ should be the bifunctor
\begin{equation}\label{eqn:Delignetensormod}
(\D^\mathrm{rev}\boxtimes\C)\times\C\rightarrow\C,\;\;\;\;
(Y\boxtimes X,X')\mapsto X\otimes X'\otimes Y.
\end{equation}

Note that if $\sigma_r:H\rightarrow K$ is a Hopf algebra surjection, then it induces naturally an injective tensor functor $\Rep(K)\hookrightarrow\Rep(H)$, which makes $\Rep(K)$ be identified as a tensor subcategory of $\Rep(H)$. Now let us establish tensor equivalences
for the relative center $\Z_{\Rep(K)}(\Rep(H))$ with the help of Theorem \ref{prop: lpd qd}:

\begin{proposition}\label{prop:relativecenterequiv}
Suppose $\sigma_r:H\rightarrow K$ is a Hopf algebra surjection.
Then there exist tensor equivalences
$$\Z_{\Rep(K)}(\Rep(H))\approx \Rep(K^{\ast\cop}\bowtie_\sigma H)\approx {}_H\YD^K.$$
\end{proposition}

\begin{proof}
According to the tensor equivalence (\ref{eqn:relativecenter-equiv}), it suffices to show that
\begin{equation}\label{eqn:QDdualtensorcat}
\big(\Rep(K)^{\mathrm{rev}}\boxtimes\Rep(H)\big)_{\Rep(H)}^\ast\approx
\Rep(K^{\ast\cop}\bowtie_\sigma H)
\end{equation}
as finite tensor categories.

Firstly, the Hopf algebra isomorphism $S^{-1}:K^\cop\cong K^\op$ induces an isomorphism $\Rep(K^\op)\cong\Rep(K^\cop)$ of finite tensor categories. Furthermore, it implies the following tensor equivalences
\begin{equation}\label{eqn:Delignetensorequiv}
\Rep(K^\op\otimes H)\cong \Rep(K^\cop\otimes H)
\approx\Rep(K^\cop)\boxtimes \Rep(H)=\Rep(K)^\mathrm{rev}\boxtimes \Rep(H),
\end{equation}
Specifically, for any $Y\in\Rep(K^\op)$ and $X\in\Rep(H)$,
the composition of (\ref{eqn:Delignetensorequiv})
sends the object $Y\otimes X$ to $Y\boxtimes X$, where the structure of $Y\in\Rep(K)^\mathrm{rev}$ should be
$$K\otimes Y\rightarrow Y,\;\;\;\;k\otimes y\mapsto S^{-1}(k)y.$$

Now we recall from Lemma \ref{lem:PAMSleft0} that the left $K^\op\otimes H$-comodule structure $\rho$ on $H$ is given as
\begin{equation*}
\rho: H\rightarrow(K^\op\otimes H)\otimes H,\;\; h\mapsto\sum(\sigma_r(S^{-1}(h_{(3)}))\otimes h_{(1)})\otimes h_{(2)}.
\end{equation*}
Thus it makes $\Rep(H)$ a left $\Rep(K^\op\otimes H)$-module category with module product bifunctor
$\ogreaterthan:\Rep(K^\op\otimes H)\times\Rep(H)\rightarrow \Rep(H)$ defined according to (\ref{eqn:Rep(H)modRep(B)}), and one can
check and obtain a natural isomorphism
$$(Y\otimes X)\ogreaterthan X'
\cong X\otimes X'\otimes Y\;\;\;\;\;\;\;\;
\big(\forall Y\in\Rep(K^\op),\;\forall X,X'\in\Rep(H)\big),$$
where the $H$-action on $X\otimes X'\otimes Y$ is diagonal:
$$H\otimes (X\otimes X'\otimes Y)\rightarrow X\otimes X'\otimes Y,\;\;
h\otimes(x\otimes x'\otimes y)\mapsto \sum h_{(1)}x\otimes h_{(2)}x'\otimes \sigma_r(S^{-1}(h_{(3)}))y.$$

However, it is direct to find that the following diagram
$$\xymatrix{
\Rep(K^\op\otimes H)\times\Rep(H)
\ar[rr]^{(\ref{eqn:Delignetensorequiv})\times\Id_{\Rep(H)}\;\;\;\;\;\;\;\;}
\ar[dr]_{\ogreaterthan}
&& \big(\Rep(K)^\mathrm{rev}\boxtimes \Rep(H)\big)\times\Rep(H)
\ar[dl]^{(\ref{eqn:Delignetensormod})} \\
& \Rep(H)   }$$
of categories and functors commutes (up to natural isomorphisms). Thus
$$\big(\Rep(K)^{\mathrm{rev}}\boxtimes\Rep(H)\big)_{\Rep(H)}^\ast
\approx \Rep(K^\op\otimes H)_{\Rep(H)}^\ast $$
as dual tensor categories. On the other hand, we combine Theorem \ref{prop: lpd qd} and the reconstruction (\ref{eqn:Doi}) of left partial duals to obtain the tensor equivalence
$\Rep(K^\op\otimes H)_{\Rep(H)}^\ast\approx\Rep(K^{\ast\cop}\bowtie_\sigma H)$. It can be concluded that (\ref{eqn:QDdualtensorcat}) holds as tensor categories.
\end{proof}

However, if the tensor functor $\D\rightarrow\C$ is not injective (or $\sigma_r$ is not surjective), we do not know how to generalize the notion $\Z_\D(\C)$ of relative center:

\begin{question}\label{question}
Let $\C$ and $\D$ be finite tensor categories. Suppose $F:\D\rightarrow\C$ is a tensor functor. What is
a ``relative center'' construction of $\C$ (with respect to $F$) which is tensor equivalent to $(\D^\mathrm{rev}\boxtimes\C)_\C^\ast$?
\end{question}

\section{Further descriptions of certain tensor equivalences}\label{Section 4}

In what follows, $K$, $H$ and $\sigma$ are as in the beginning of Section \ref{Section3}.


\subsection{Tensor categories of two-sided two-cosided relative Hopf modules}\label{subsection 4.1}
Our main purpose in this subsection is to introduce certain equivalences between two tensor categories.

The first category
${}^K_K\M^K_H$ consists of finite-dimensional vector spaces $M$ which are $K$-$H$-bimodules and $K$-$K$-bicomodules satisfying that both comodule structures on $M$ preserve both of its module structures. Specifically, for any $k\in K$, $h\in H$ and $m\in M$, the following compatibility conditions hold:
\begin{equation}\label{eqn: ll condition}
\sum(k\cdot m)^{(-1)}\otimes(k\cdot m)^{(0)}=
\sum k_{(1)}m^{(-1)}\otimes k_{(2)}\cdot m^{(0)},
\end{equation}
\begin{equation}\label{eqn: rl condition}
\sum(m\cdot h)^{(-1)}\otimes (m\cdot h)^{(0)}=
\sum m^{(-1)}\sigma_r(h_{(1)})\otimes m_{(0)}\cdot h_{(2)},
\end{equation}
\begin{equation}\label{eqn: lr condition}
\sum(k\cdot m)^{(0)}\otimes(k\cdot m)^{(1)}=
\sum k_{(1)}\cdot m^{(0)}\otimes k_{(2)}m^{(1)},
\end{equation}
\begin{equation}\label{eqn: rr condition}
\sum(m\cdot h)^{(0)}\otimes (m\cdot h)^{(1)}=
\sum m^{(0)}\cdot h_{(1)}\otimes m^{(1)}\sigma_r(h_{(2)}).
\end{equation}
Here $m\mapsto\sum m^{(-1)}\otimes m^{(0)}$ and $m\mapsto\sum m^{(0)}\otimes m^{(1)}$ denote respectively the left and right $K$-comodules structures on $M$.

The other category
${}^{K^\ast}_{K^\ast}\M^{ H^\ast}_{K^\ast}$ is defined similarly via the Hopf algebra map $\sigma_l:K^\ast\rightarrow H^\ast$, where the comodule structures on its objects are denoted with superscript parentheses as well. Furthermore, both of them can become tensor categories.

\begin{lemma}\label{lem: tensor cat}
With above notation,
\begin{itemize}
\item[\rm{(1)}]
${}^K_K\M^ K_H$ is a finite tensor category with tensor product bifunctor $\square_K$ and unit object $K$. Specifically, for any $M,N\in{}^K_K\M^ K_H$,
\begin{itemize}
\item
The left $K$-action and right $H$-action on $M\square_K N$ are diagonal;
\item
The left and right $K$-coactions on $M\square_K N$ are determined at the first and second (co)tensorands respectively.
\end{itemize}

\item[\rm{(2)}]
${}^{K^\ast}_{K^\ast}\M^{ H^\ast}_{K^\ast}$ is a finite tensor category with tensor product bifunctor $\otimes_{K^\ast}$ and unit object $K^\ast$. Specifically, for any $M, N\in {}^{K^\ast}_{K^\ast}\M^{ H^\ast}_{K^\ast}$,
\begin{itemize}
\item
The left $K^\ast$-coaction and right $H^\ast$-coaction on $M\otimes_{K^\ast} N$ are diagonal;
\item
The left and right $K^\ast$-actions on $M\otimes_{K^\ast} N$ are determined at the first and second tensorands respectively.
\end{itemize}

\end{itemize}
\end{lemma}

\begin{proof}
\begin{itemize}
\item[(1)]

Firstly, we know by the definition of cotensor products that
\begin{equation}\label{eqn: cotp MN}
M\square_KN=\Big\{\sum_im_i\otimes n_i\in M\otimes N \Bigm|
\sum_im_i^{(0)}\otimes m_i^{(1)}\otimes n_i
=\sum_im_i\otimes n_i^{(-1)}\otimes n_i^{(0)}\Big\}.
\end{equation}
It is direct to show that the right diagonal $H$-action is closed on $M\square_KN$. Namely, for any $h\in H$
and $\sum_im_i\otimes n_i\in M\square_K N$, a computation using (\ref{eqn: rl condition}) and (\ref{eqn: rr condition}) shows that
$\sum_i m_i\cdot h_{(1)}\otimes n_i\cdot h_{(2)}$ belongs to $M\square_KN$ as follows:
Similarly, the left diagonal $K$-action is also closed on $M\square_KN$, and one can easily conclude that $M\square_K N$ becomes a $K$-$H$-bimodule via diagonal actions.

On the other hand, it follows from \cite[Introduction]{Tak77} that $M\square_KN$ admits the canonical $K$-$K$-bicomodule structure as claimed.
Furthermore, it is straightforward to verify by (\ref{eqn: ll condition})-(\ref{eqn: rr condition})
 that both comodule structures on $M\square_KN$ preserve both of its module structures.

%

Finally, it is clear that $K$ is the unit object, and the canonical isomorphisms
$$(M\square_KN)\square_KP\cong M\square_K(N\square_KP)\;\;\;\;\text{and}\;\;\;\;
K\square_KM\cong M\cong M\square_KK$$
can be found in \cite[Section 0]{Tak77}, which are natural in $M,N,P\in {}^K_K\M^ K_H$.

\item[(2)]
At first let us show that the right diagonal $H^\ast$-coaction on $M\otimes_{K^\ast} N$ is well-defined. For the purpose, we need to check that for any $m\in M$, $n\in N$ and $k^\ast\in K^\ast$,
the images of the elements
$$m\cdot k^\ast\otimes_{K^\ast} n\;\;\;\;\text{and}\;\;\;\;
m\otimes_{K^\ast} k^\ast\cdot n$$
are equal under the right diagonal $H^\ast$-comodule structure.
Indeed, one can compute to find
\begin{eqnarray*}
\sum\big[(m\cdot k^\ast)^{(0)}\otimes_{K^\ast} n^{(0)} \big]\otimes (m\cdot k^\ast)^{(1)}n^{(1)}
=
\sum\big[m^{(0)}\otimes_{K^\ast} (k^\ast\cdot n)^{(0)} \big]\otimes m^{(1)}(k^\ast\cdot n)^{(1)}.
\end{eqnarray*}

Similar arguments imply that the left diagonal $K^\ast$-coaction on $M\otimes_{K^\ast} N$ is also well-defined. One can finally conclude that
$M\otimes_{K^\ast} N$ is endowed with the $K^\ast$-$K^\ast$-bimodule and $K^\ast$-$H^\ast$-bicomodule structures as desired.

Analogously to the proof of (1), here we verify the compatibility of the right $H^\ast$-comodule structure and the right $K^\ast$-module structure on $M^\ast\otimes_{K^\ast} N$ for instance: For any $m\in M$, $n\in N$ and $k^\ast\in K^\ast$, we have
\begin{eqnarray*}
\sum\big[m^{(0)}\otimes(n\cdot k^\ast)^{(0)}\big]\otimes m^{(1)}(n\cdot k^\ast)^{(1)}&=&\sum (m^{(0)}\otimes_{K^\ast} n^{(0)}\cdot k^\ast_{(1)})\otimes m^{(1)}n^{(1)}\sigma_l(k^\ast_{(2)}) \\
&=&\sum [(m^{(0)}\otimes_{K^\ast}n^{(0)})\otimes m^{(1)}n^{(1)}]\cdot k^\ast.
\end{eqnarray*}

Clearly, the monoidal category ${}^{K^\ast}_{K^\ast}\M^{ H^\ast}_{K^\ast}$ has unit object $K^\ast$ and canonical isomorphisms
$$(M\otimes_{K^\ast}N)\otimes_{K^\ast}P\cong M\otimes_{K^\ast}(N\otimes_{K^\ast}P)\;\;\;\;\text{and}\;\;\;\;
K^\ast\otimes_{K^\ast}M\cong M\cong M\otimes_{K^\ast}K^\ast,$$
which are natural in $M,N,P\in {}^{K^\ast}_{K^\ast}\M^{ H^\ast}_{K^\ast}$.
\end{itemize}
\end{proof}

In fact, the two tensor categories ${}^K_K\M^K_H$ and ${}^{K^\ast}_{K^\ast}\M^{ H^\ast}_{K^\ast}$ are equivalent via the duality functors.
For convenience, we still use $\rightharpoonup$ and $\leftharpoonup$ without confusions, to denote the left $H$-action and right $K$-action on each $M\in{}^{K^\ast}_{K^\ast}\M^{ H^\ast}_{K^\ast}$ (induced respectively by its right $H^\ast$-comodule and left $K^\ast$-comodule structures) as follows:
\begin{equation}\label{eqn: hit action}
h\rightharpoonup m=\sum m^{(0)}\la m^{(1)}, h\ra
\;\;\;\;\text{and}\;\;\;\;
m\leftharpoonup k=\sum\la m^{(-1)}, k\ra m^{(0)}
\end{equation}
for any $k\in K$, $h\in H$ and $m\in M$.

\begin{proposition}\label{prop: four* cong four}
There is a contravariant equivalence
\begin{equation}\label{eqn: four* cong four}
{}^{K^\ast}_{K^\ast}\M^{ H^\ast}_{K^\ast}\approx{}^K_K\M^K_H,\;\;\;\;M\mapsto M^\ast
\end{equation}
between finite tensor categories introduced in Lemma \ref{lem: tensor cat}, with
the monoidal structure 
\begin{equation}\label{J_{M, N}}
J_{M,N}:M^\ast\square_K N^\ast \rightarrow (M\otimes_{K^\ast}N)^\ast,
\;\;\;\;\sum_i m_i^\ast\otimes n_i^\ast\mapsto\sum_i\big\langle m_i^\ast, -\big\rangle\big\langle n_i^\ast, -\big\rangle,
\end{equation}
and the quasi-inverse
\begin{equation}\label{eqn: four* cong four inverse}
P^\ast\mapsfrom P.
\end{equation}
\end{proposition}

\begin{proof}
At first
for each $M\in{}^{K^\ast}_{K^\ast}\M^{ H^\ast}_{K^\ast}$, we set $M^\ast$ as the object in ${}^K_K\M^ K_H$ with four structures induced canonically as follows:
The left $K$-action and the right $H$-action on $M^\ast$ are respectively given by
\begin{equation}\label{eqn: mod M*}
 k\cdot m^\ast=\langle m^\ast, (-)\leftharpoonup k\rangle\;\;\;\;\text{and}\;\;\;\;m^\ast\cdot h=\langle m^\ast h\rightharpoonup(-)\rangle,
\end{equation}
for any $k\in K$, $h\in H$ and $m^\ast\in M^\ast$,
which make $M^\ast$ a $K$-$H$-bimodule.
On the other hand,
the left and right $K$-coactions
\begin{equation}\label{eqn: comod M*}
m^\ast\mapsto\sum m^{\ast(-1)}\otimes m^{\ast(0)}\;\;\;\text{and}\;\;\;
m^\ast\mapsto\sum m^{\ast(0)}\otimes m^{\ast(1)}
\end{equation}
on $M^\ast$ are determined such that the equations
\begin{equation}\label{eqn: lrcomod M* relation}
\sum\langle k^\ast, m^{\ast(-1)}\rangle\langle m^{\ast(0)}, m\rangle=\langle m^\ast, k^\ast\cdot m\rangle,
\;\;\;\;
\sum\langle m^{\ast(0)}, m\rangle\langle k^\ast, m^{\ast(1)}\rangle=\langle m^\ast, m\cdot k^\ast\rangle
\end{equation}
hold for any $k^\ast\in K^\ast$ and $m\in M$. It is clear that
these coactions equip $M^\ast$ with
the structure of a $K$-$K$-bicomodule.

Here we only verify that the left $K$-comodule structure and the right $H$-module structure of $M^\ast$ satisfy the compatibility conditions (\ref{eqn: rl condition}) in the category ${}^K_K\M^ K_H$ as an example. In order to show that
$$\sum(m^\ast\cdot h)^{(-1)}\otimes(m^\ast\cdot h)^{(0)}
=\sum m^{\ast(-1)}\sigma_r(h_{(1)})\otimes m^{\ast(0)}\cdot h_{(2)}$$
holds
for each $h\in H$ and $m^\ast\in M^\ast$, we compare the images of both sides under any $k^\ast\otimes m$ ($k^\ast\in K^\ast,\;m\in M$) by
the following calculations:
\begin{eqnarray*}
&&
\sum\langle k^\ast, (m^\ast\cdot h)^{(-1)}\rangle\langle (m^\ast\cdot h)^{(0)}, m\rangle
\overset{(\ref{eqn: lrcomod M* relation})}=
\la m^\ast\cdot h, k^\ast\cdot m\ra \overset{(\ref{eqn: mod M*})}=
\langle m^\ast, h\rightharpoonup (k^\ast\cdot m)\rangle  \\
&\overset{(\ref{eqn: hit action})}=&
\sum\langle m^\ast, k^\ast_{(1)}\cdot m_{(0)}\rangle\langle\sigma_l(k^\ast_{(2)})m_{(1)}, h\rangle \\
&\overset{(\ref{eqn: lrcomod M* relation})}=&
\sum\langle k^\ast_{(1)}, m^{\ast(-1)}\rangle \langle k^\ast_{(2)}, \sigma_r(h_{(1)})\rangle\langle m^{\ast(0)}, m_{(0)}\rangle\langle m_{(1)}, h_{(2)}\rangle \\
&\overset{(\ref{eqn: hit action})}=&
\sum\langle k^\ast, m^{\ast(-1)}\sigma_r(h_{(1)})\rangle\langle m^{\ast(0)}, h_{(2)}\rightharpoonup m\rangle \\
&\overset{(\ref{eqn: mod M*})}=&
\sum\langle k^\ast, m^{\ast(-1)}\sigma_r(h_{(1)})\rangle\langle m^{\ast(0)}\cdot h_{(2)}, m\rangle,
\end{eqnarray*}
where the fourth equality is due to
$\sum(k^\ast\cdot m)^{(0)}\otimes(k^\ast\cdot m)^{(1)}
=\sum (k^\ast_{(1)}\cdot m_{(0)})\otimes\sigma_l(k^\ast_{(2)})m_{(1)}$ according to the assumption $M\in{}^{K^\ast}_{K^\ast}\M^{ H^\ast}_{K^\ast}$.
Moreover, we conclude by analogous processes that $M^\ast$ is an object in ${}^K_K\M^ K_H$, and hence $M\mapsto M^\ast$ is a well-defined functor with quasi-inverse $P^\ast\mapsfrom P$ evidently.

Next, we try to prove that $J$ (\ref{J_{M, N}}) is a well-defined natural isomorphism. To this end, it should be verified that
$\sum_i\big\langle m_i^\ast, -\big\rangle\big\langle n_i^\ast, -\big\rangle$ should be a well-defined function on $M\otimes_{K^\ast}N$ for each
$\sum_i m_i^\ast\otimes n_i^\ast\in M\square_KN$,
which is due to following calculations: For any $k^\ast\in K^\ast$, $m\in M$ and $n\in N$,
\begin{eqnarray*}
&& \big\la J_{M, N}\big(\sum_i m_i^\ast\otimes n_i^\ast\big),
m\cdot k^\ast\otimes_{K^\ast} n\big\ra
\overset{(\ref{eqn: lrcomod M* relation})}=
\sum_i\la m_i^{\ast(0)}, m\ra\la k^\ast, m_i^{\ast(1)}\ra\la n_i^\ast, n\ra \\
&\overset{(\ref{eqn: cotp MN})}=&
\sum_i\la m_i^\ast, m\ra\la k^\ast, n_i^{\ast(-1)}\ra\la n_i^{\ast(0)}, n\ra
\overset{(\ref{eqn: lrcomod M* relation})}=
\sum_i\la m_i^\ast, m\ra\la n_i^\ast, k^\ast\cdot n\ra \\
&=&
\big\la J_{M, N}\big(\sum_i m_i^\ast\otimes n_i^\ast\big),m\otimes_{K^\ast} k^\ast\cdot n\big\ra.
\end{eqnarray*}

In fact, it is known that $M\otimes_{K^\ast}N$ is the coequalizer of the diagram
$$\xymatrix{
M\otimes K^\ast\otimes N
\ar@<-0.5ex>[rr]_{\;\;\;\;\id_{M}\otimes \mu_N}  \ar@<.5ex>[rr]^{\;\;\;\;\nu_M\otimes\id_{N}}
&& \;M\otimes N\; \ar@{->>}[r]
& M\otimes_{K^\ast}N  },$$
where $\mu$ and $\nu$ denote respectively the left $K^\ast$-module and right $K^\ast$-module structures of objects in ${}^{K^\ast}_{K^\ast}\M^{ H^\ast}_{K^\ast}$.
Then it is sent by the exact functor $(-)^\ast$ to the diagram
$$\xymatrix{
(M\otimes_{K^\ast}N)^\ast\;\; \ar@{>->}[r] & \;M^\ast\otimes N^\ast\;
\ar@<-0.5ex>[rr]_{\id_{M^\ast}\otimes\mu_N^\ast\;\;\;\;}  \ar@<.5ex>[rr]^{\nu_M^\ast\otimes\id_{N^\ast}\;\;\;\;}
&& \;M^\ast\otimes K\otimes N^\ast   }.$$
However, one can find that $\mu_N^\ast$ and $\nu_M^\ast$ coincide with the induced $K$-module structures of $M^\ast$ and $N^\ast$ respectively.
Thus according to the definition of cotensor product in \cite[Section 0]{Tak77}, it is clear that $M^\ast\square_KN^\ast\cong(M\otimes_{K^\ast}N)^\ast$ as the equalizer of the diagram above, and this isomorphism is exactly $J_{M,N}$.

Besides, $J_{M,N}$ is evidently natural in $M,N\in {}^{K^\ast}_{K^\ast}\M^{ H^\ast}_{K^\ast}$, and now we need to show that it is a morphism in ${}^K_K\M^ K_H$. Let us verify that $J_{M, N}$ preserves left $K$-actions for instance, while others are completely analogous as well: For any $m\in M$, $n\in N$ and $k\in K$,
\begin{eqnarray*}
&& \Big\la J_{M, N}\Big(\sum_ik\cdot (m_i^\ast\otimes n_i^\ast)\Big), m\otimes_{K^\ast}n\Big\ra   \\
&\overset{(\ref{J_{M, N}})}=&
\sum_i\langle k_{(1)}\cdot m_i^\ast, m\rangle\langle k_{(2)}\cdot n_i^\ast, n\rangle
\overset{(\ref{eqn: mod M*})}=
\sum_i\langle m_i^\ast, m\leftharpoonup k_{(1)}\rangle\langle n_i^\ast, n\leftharpoonup k_{(2)}\rangle \\
&\overset{(\ref{eqn: hit action})}=&
\sum_i\langle m^{(-1)}, k_{(1)}\rangle\langle m_i^\ast, m^{(0)}\rangle\langle n^{(-1)}, k_{(2)}\rangle\langle n_i^\ast, n^{(0)}\rangle \\
&\overset{(\ref{J_{M, N}}),\;(\ref{eqn: hit action})}=&
\sum_i\big\langle J_{M, N}\big(\sum_i m_i^\ast\otimes n_i^\ast\big), (m\otimes_{K^\ast} n)\leftharpoonup k\big\rangle \\
&\overset{(\ref{eqn: mod M*})}=&
\big\la k\cdot J_{M, N}\big(\sum_i m_i^\ast\otimes n_i^\ast\big), m\otimes_{K^\ast} n\big\ra,
\end{eqnarray*}
where the penultimate equality is because the left $K^\ast$-coaction on $M\otimes_{K^\ast}N$ is diagonal.

Finally, it suffices to show the equation
\begin{equation}\label{eqn: JJ=JJ}
J_{M\otimes_{K^\ast}N,P}\circ(J_{M,N}\otimes\id_{P^\ast})
=J_{M,N\otimes_{K^\ast}P}\circ(\id_{M^\ast}\otimes J_{N,P})
\end{equation}
holds for any $M,N,P\in{}^{K^\ast}_{K^\ast}\M^{ H^\ast}_{K^\ast}$. This is because the images of every element $\sum_i m_i^\ast\otimes n_i^\ast\otimes p_i^\ast\in M^\ast\square_KN^\ast\square_KP^\ast$ under the left and right sides of (\ref{eqn: JJ=JJ}) are both calculated to be
$\sum_i \la m_i^\ast,-\ra\la n_i^\ast,-\ra\la p_i^\ast,-\ra$.
\end{proof}

\begin{remark}\label{rmk: four cong four*}
We describe the quasi-inverse (\ref{eqn: quasi-inverse}) in details for subsequent uses.

For each $P\in{}^K_K\M^ K_H$, we set $P^\ast\in{}^{K^\ast}_{K^\ast}\M^{ H^\ast}_{K^\ast}$ with four structures also induced canonically as follows:
The left and right $K^\ast$-actions are respectively given by
\begin{equation}\label{eqn: mod P*}
 k^\ast\cdot p^\ast=\langle p^\ast, (-)\leftharpoonup k^\ast\rangle
 \;\;\;\;\text{and}\;\;\;\;
 p^\ast\cdot k^\ast=\langle p^\ast, k^\ast\rightharpoonup(-)\rangle
\end{equation}
for any $k^\ast\in K^\ast$ and $p^\ast\in P^\ast$.
On the other hand,
the left $K^\ast$-coaction and right $H^\ast$-coaction
\begin{equation}\label{eqn: comod P*}
p^\ast\mapsto\sum p^{\ast(-1)}\otimes p^{\ast(0)}\;\;\;\text{and}\;\;\;
p^\ast\mapsto\sum p^{\ast(0)}\otimes p^{\ast(1)},
\end{equation}
are determined such that the equations
\begin{equation}\label{eqn: lrcomod P* relation}
\sum\langle p^{\ast(-1)},k\rangle\langle p^{\ast(0)}, p\rangle=\langle p^\ast, k\cdot p\rangle,
\;\;\;\;
\sum\langle p^{\ast(0)}, p\rangle\langle p^{\ast(1)},h\rangle=\langle p^\ast, p\cdot h\rangle
\end{equation}
hold for any $k\in K$, $h\in H$ and $p\in P$.
\end{remark}

\subsection{Tensor equivalences between the various categories}
In this subsection, we apply the results of Section \ref{Section3} to provide further tensor equivalences of the categories mentioned in the previous sections.

Note in Proposition \ref{prop: te lDoi} that ${}^{}_{K^{\ast\cop}}\M^{K^{\ast\cop}\otimes H^\ast}_{K^{\ast\cop}}$ is a finite tensor category, whose structure is defined according to \cite[Proposition 4.7]{Li23}. Specifically,
for any $M,N\in{}^{}_{K^{\ast\cop}}\M^{K^{\ast\cop}\otimes H^\ast}_{K^{\ast\cop}}$, their tensor product object $M\otimes_{K^{\ast\cop}} N$ has structures as follows:
\begin{itemize}
\item
The left and right $K^{\ast\cop}$-actions are determined at the first and second tensorands respectively;

\item
The right $K^{\ast\cop}\otimes H^\ast$-coaction is diagonal:
\begin{equation}
\begin{array}{ccc}
M\otimes_{K^{\ast\cop}} N & \rightarrow &
\big(M\otimes_{K^{\ast\cop}} N\big)\otimes (K^{\ast\cop}\otimes H^\ast), \\
m\otimes_{K^{\ast\cop}} n &\mapsto&
\sum(m_{(0)}\otimes_{K^{\ast\cop}} n_{(0)})\otimes m_{(1)}n_{(1)}
\end{array}
\end{equation}
\end{itemize}

On the other hand, recall in Lemma \ref{lem: tensor cat} we have established the structures of  ${}^{K^\ast}_{K^\ast}\M^{ H^\ast}_{K^\ast}$  as a finite tensor category, which is indeed isomorphic to the previous one.

\begin{lemma}\label{lem: four* cong three}
There is an isomorphism of finite tensor categories
\begin{equation}\label{eqn: four* cong three}
{}^{K^\ast}_{K^\ast}\M^{ H^\ast}_{K^\ast}\cong{}^{}_{K^{\ast\cop}}\M^{K^{\ast\cop}\otimes H^\ast}_{K^{\ast\cop}}.
\end{equation}
\end{lemma}

\begin{proof}
At first for each $M\in {}^{K^\ast}_{K^\ast}\M^{ H^\ast}_{K^\ast}$ , we set $M$ as the object in ${}^{}_{K^{\ast\cop}}\M^{K^{\ast\cop}\otimes H^\ast}_{K^{\ast\cop}}$ with three structures induced as follows:
The $K^{\ast\cop}$-$K^{\ast\cop}$-bimodule structure on $M$ coincides with its original $K^\ast$-$K^\ast$-bimodule structure, and the right $K^{\ast\cop}\otimes H^\ast$-comodule structure is defined as
\begin{equation}\label{eqn: rK*coptH*com of M}
 M\rightarrow  M\otimes(K^{\ast\cop}\otimes H^\ast),\;\;\;m\mapsto \sum m^{(0)}\otimes (m^{(-1)}\otimes m^{(1)}).
\end{equation}
By direct verification
 using (\ref{eqn: rightcoidealsubalg}), one can find that the coaction (\ref{eqn: rK*coptH*com of M}) preserves both left and right $K^{\ast\cop}$-actions.
It follows that $M\in{}^{}_{K^{\ast\cop}}\M^{K^{\ast\cop}\otimes H^\ast}_{K^{\ast\cop}}$, and thus we obtain the desired functor, which is clearly an isomorphism.

Now we explain that the isomorphism defined above is a tensor functor.
In fact,
for any $M, N\in {}^{K^\ast}_{K^\ast}\M^{ H^\ast}_{K^\ast}$,
their tensor product object $M\otimes_{K^\ast}N$
will be sent to the relative Doi-Hopf module $M\otimes_{K^{\ast\cop}}N$ as an object in ${}^{}_{K^{\ast\cop}}\M^{K^{\ast\cop}\otimes H^\ast}_{K^{\ast\cop}}$. This is because $M\otimes_{K}N$ and $M\otimes_{K^{\ast\cop}}N$ are in fact the same $K^\ast$-$K^\ast$-bimodules (or equivalently, $K^{\ast\cop}$-$K^{\ast\cop}$-bimodules), and their right $K^{\ast\cop}\otimes H^\ast$-coactions are both induced to be
\begin{equation}
m\otimes_{K^\ast} n\mapsto\sum(m^{(0)}\otimes n^{(0)})\otimes(m^{(-1)}n^{(-1)}\otimes m^{(1)}n^{(1)}).
\end{equation}
Besides, the unit object $K^\ast$ is also sent to the unit object $K^{\ast\cop}$. As a conclusion, we have defined a tensor isomorphism with identity monoidal structure.
\end{proof}

In particular, for each $V\in\Rep(K^{\ast\cop}\bowtie_\sigma H)$,
 we note that $V\otimes K^\ast$ can be an object in ${}^{}_{K^{\ast\cop}}\M^{K^{\ast\cop}\otimes H^\ast}_{K^{\ast\cop}}$ according to Lemma $\ref{lem:leftcoiso}$. By the result of the lemma above, it follows immediately that $V\otimes K^\ast$ also belongs to the category ${}^{K^\ast}_{K^\ast}\M^{ H^\ast}_{K^\ast}$.
Specifically, one may verify that the left $K^\ast$-comodule and right $H^\ast$-comodule structures on $V\otimes K^\ast$ are respectively given as:
\begin{equation}\label{eqn: lK* rH*}
v\otimes k^\ast\mapsto\sum k^\ast_{(1)}\otimes (v\otimes k^\ast_{(2)})\;\;\;\;\text{and}\;\;\;\;
v\otimes k^\ast\mapsto\sum (v_{\langle0\rangle}\otimes k^\ast_{(1)})\otimes v_{\langle1\rangle}\sigma_l(k^\ast_{(2)}),
\end{equation}
where $\sum v_{\la0\ra}\la l^\ast,v_{\la1\ra}\ra=(l^\ast\bowtie1)v$ holds for any $l^\ast\in K^\ast$.
In fact, these structures
will induce via (\ref{eqn: rK*coptH*com of M}) the right $K^{\ast\cop}\otimes H^\ast$-comodule structure on $V\otimes K^\ast$ as (\ref{comods of VotimesK*}).
On the other hand,
the left and right $K^{\ast}$-actions on $V\otimes K^\ast$ coincide in fact with (\ref{mods of VotimesK*}) given by:
\begin{equation}\label{mods of VotimesK*new}
l^\ast\cdot(v\otimes k^\ast)=\sum (l^\ast_{(2)}\bowtie v)\otimes l^\ast_{(1)}k^\ast\;\;\;\;\text{and}\;\;\;\;(v\otimes k^\ast)\cdot l^\ast=\sum v\otimes k^\ast l^\ast
\end{equation}
for any $l^\ast\in K^{\ast}$, $v\in V$ and $k^\ast\in K^\ast$.

Consequently, we know by Proposition \ref{prop: four* cong four} that $V^\ast\otimes K\cong(V\otimes K^\ast)^\ast\in{}^K_K\M^ K_H$.
Therefore, through the composition of the tensor equivalences between the previously mentioned categories, we conclude the main theorem of this section as follows.

\begin{theorem}\label{them: cc}
There are (covariant) tensor equivalences of finite tensor categories:
\begin{equation}\label{eqn: four equiv HYDK}
\big({}^K_K\M^K_H\big)^\vee\approx{}^{K^\ast}_{K^\ast}\M^{ H^\ast}_{K^\ast}\approx\Rep(K^{\ast\cop}\bowtie_\sigma H)\cong{}_H\mathfrak{YD}^K,
\end{equation}
whose composition is
\begin{equation}\label{eqn: equivalence}
M\mapsto\overline{M^\ast}=M^\ast/(M^\ast\cdot(K^\ast)^+)
\end{equation}
with quasi-inverse
\begin{equation}\label{eqn: quasi-inverse}
V^\ast\otimes K\mapsfrom V.
\end{equation}
Here $(-)^\vee$ denotes the category with reversed arrows.
\end{theorem}

\begin{proof}
According to our preceding results,
we start by describing the three equivalences in (\ref{eqn: four equiv HYDK}), and show that their composition will be (\ref{eqn: equivalence}) as a result:
\begin{itemize}
\item[(1)]
The first one is established in Proposition \ref{prop: four* cong four}, which sends each $M\in{}^K_K\M^ K_H$ to its dual space $M^\ast$ as an object in ${}^{K^\ast}_{K^\ast}\M^{ H^\ast}_{K^\ast}$;

\item[(2)]
The second functor is the composition of the isomorphism in Lemma \ref{lem: four* cong three} and $\Phi$ in Corollary \ref{cor: te lDoi},  and it sends the object $M^\ast\in{}^{K^\ast}_{K^\ast}\M^{ H^\ast}_{K^\ast}$ to the quotient $$M^\ast/(M^\ast\cdot(K^{\cop\ast})^+)\;\;\;\;\;\;\text{(or}\;\; M^\ast/(M^\ast\cdot(K^{\ast})^+)\;\;
\text{without confusion)} $$
in $\Rep(K^{\ast\cop}\bowtie_\sigma H)$;

\item[(3)]
The last equivalence is from Corollary \ref{cor: YD iso repqd}, making $M^\ast/(M^\ast\cdot(K^{\ast})^+)$ be endowed with the structures of a relative Yetter-Drinfeld module in ${}_H\mathfrak{YD}^K$.
\end{itemize}

Conversely, due to similar arguments, the composition of quasi-inverses of (\ref{eqn: four equiv HYDK})
sends each $V\in{}_H\mathfrak{YD}^K$ to the object of
the form $(V\otimes K^\ast)^\ast$ in ${}^K_K\M^K_H$.
Specifically, $V$ should be at first a left $K^{\ast\cop}\bowtie_\sigma H$-module via the isomorphism (\ref{eqn: YD iso repqd}), and thus
$V\otimes K^\ast\in{}^{K^\ast}_{K^\ast}\M^{ H^\ast}_{K^\ast}$ as explained before this theorem, whose coactions are given by (\ref{eqn: lK* rH*}).
Then it is sent by (\ref{eqn: four* cong four}) to the dual space $(V\otimes K^\ast)^\ast$ with structures determined via (\ref{eqn: mod M*}) and (\ref{eqn: comod M*}).

Now we define on the space $V^\ast\otimes K$ four structures as follows:
The left $K$-action and the right $H$-action are respectively given by:
\begin{equation}\label{eqn: modstru of V*K}
l\cdot(v^\ast\otimes k)=v^\ast\otimes lk\;\;\;\;
\text{and}\;\;\;\;
(v^\ast\otimes k)\cdot h=\sum (v^\ast\cdot  h_{(1)})\otimes k\sigma_r(h_{(2)}),
\end{equation}
for any $l,k\in K$ and $v^\ast\in V^\ast$, which make $V^\ast\otimes K$ a $K$-$H$-bimodule.
On the other hand, the left and right $K$-coactions
\begin{equation}\label{eqn: comodstru of V*K}
v^\ast\otimes k\mapsto\sum k_{(1)}v^\ast_{\langle-1\rangle}\otimes(v^\ast_{\langle0
\rangle}\otimes k_{(2)})\;\;\;\;\text{and}\;\;\;\;
v^\ast\otimes k\mapsto\sum(v^\ast\otimes k_{(1)})\otimes k_{(2)}
\end{equation}
on $V^\ast\otimes K$ are defined, where
\begin{equation}\label{eqn: v0v1}
\sum v^\ast_{\langle-1\rangle}\langle v^\ast_{\langle0\rangle}, v\rangle=\sum\langle v^\ast, v_{\langle0\rangle}\rangle v_{\langle1\rangle}
\end{equation}
holds for any $v\in V$.

To complete the proof, we claim that
$V^\ast\otimes K$ is an object isomorphic to $(V\otimes K^\ast)^\ast$ in ${}^K_K\M^K_H$. For the purpose, it suffices to verify that
the canonical linear isomorphism
$\varphi:V^\ast\otimes K\cong (V\otimes K^\ast)^\ast$
preserves both actions and both coactions by following calculations:
\begin{itemize}
\item
$\varphi$ is a right $H$-module map, since
\begin{eqnarray*}
&& \la\varphi(v^\ast\otimes k)\cdot h, v\otimes k^\ast\ra
\overset{(\ref{eqn: mod M*})}=
\la \varphi(v^\ast\otimes k),h\rightharpoonup(v\otimes k^\ast)\ra  \\
&\overset{(\ref{eqn: lK* rH*})}=&
\sum\la \varphi(v^\ast\otimes k),v_{\la0\ra}\otimes k^\ast_{(1)}\ra
  \la v_{\la1\ra}\sigma_l(k^\ast_{(2)}),h\ra  \\
&=&
\sum\la v^\ast,v_{\la0\ra}\ra\la k^\ast_{(1)},k\ra\la v_{\la1\ra},h_{(1)}\ra
  \la k^\ast_{(2)},\sigma_r(h_{(2)})\ra  \\
&\overset{(\ref{eqn: v0v1})}=&
\sum \la v^\ast\cdot h_{(1)},v\ra\la k^\ast,k\sigma_r(h_{(2)})\ra
=
\sum \la\varphi((v^\ast\cdot h_{(1)})\otimes k\sigma_r(h_{(2)})),v\otimes k^\ast\ra
\\
&\overset{(\ref{eqn: modstru of V*K})}=&
\la \varphi((v^\ast\otimes k)\cdot h), v\otimes k^\ast\ra
\end{eqnarray*}
hold for all $h\in H$, $v^\ast\in V^\ast$, $k\in K$, $v\in V$ and $k\in K^\ast$.
The verification that $\varphi$ preserves the left $H$-actions is analogous.
\item
$\varphi$ is a left $K$-comodule map, which means by (\ref{eqn: comodstru of V*K}) that
$$\sum k_{(1)} v^\ast_{\la-1\ra}\otimes\varphi(v^\ast_{\la0\ra}\otimes k_{(2)})
=\sum\varphi(v^\ast\otimes k)^{(-1)}\otimes\varphi(v^\ast\otimes k)^{(0)},$$
for all $v^\ast\in V^\ast$ and $k\in K$, since
\begin{eqnarray*}
&& \sum \la l^\ast,k_{(1)} v^\ast_{\la-1\ra}\ra\la\varphi(v^\ast_{\la0\ra}\otimes k_{(2)}),v\otimes k^\ast\ra
\overset{(\ref{eqn: v0v1})}=
\sum \la v^\ast,v_{\la0\ra}\ra \la l^\ast_{(2)},v_{\la1\ra}\ra\la l^\ast_{(1)}k^\ast,k\ra  \\
&\overset{(\ref{eqn: YD iso repqd action})}=&
\sum\la v^\ast,(l^\ast_{(2)}\bowtie1)v\ra\la l^\ast_{(1)}k^\ast,k\ra
=
\sum\la \varphi(v^\ast\otimes k),(l^\ast_{(2)}\bowtie1)v\otimes l^\ast_{(1)}k^\ast\ra  \\
&\overset{(\ref{mods of VotimesK*new})}=&
\sum\la \varphi(v^\ast\otimes k),l^\ast\cdot (v\otimes k^\ast)\ra
\overset{(\ref{eqn: lrcomod M* relation})}=
\sum\la l^\ast,\varphi(v^\ast\otimes k)^{(-1)}\ra
  \la\varphi(v^\ast\otimes k)^{(0)},v\otimes k^\ast\ra
\end{eqnarray*}
hold for all $v^\ast\in V^\ast$, $k\in K$, $v\in V$ and $k^\ast,l^\ast\in K^\ast$.
The verification that $\varphi$ preserves the right $K$-coactions is analogous.
\end{itemize}
As a conclusion, $V^\ast\otimes K$ belongs to ${}^K_K\M^K_H$ as well, and $V^\ast\otimes K\cong (V\otimes K^\ast)^\ast$ is an isomorphism in ${}^K_K\M^K_H$ which is natural in $V$. Therefore, (\ref{eqn: quasi-inverse}) is also a quasi-inverse of (\ref{eqn: equivalence}).
\end{proof}

\subsection{Comparison with Schauenburg's characterization}

Since the antipode of a finite-dimensional Hopf algebra is bijective according to \cite[Proposition 2]{LS69},
we cite Schauenburg's characterization \cite[Corollary 6.4]{Sch94} in finite-dimensional cases as the following lemma.

\begin{lemma}\label{lem: sch94}
There is an equivalence of finite tensor categories
\begin{equation*}
{}_H^H\M_H^H\approx{}^{H}\mathfrak{YD}_{H},\;\;\; M\mapsto M_\mathrm{coinv},
\end{equation*}
which sends each $M\in{}_K^K\M_H^K$ to the space $M_\mathrm{coinv}$ of its coinvariants as a right $H$-comodule, with structures as follows:
\begin{itemize}
\item
The right $H$-module structure $\triangleleft$ is given by
\begin{equation}\label{eqn: rH of McoK}
m\triangleleft h=\sum S^{-1}(h_{(2)})\cdot m\cdot h_{(1)}\;\;\;\;\;\;\;\;
(\forall m\in M_\mathrm{coinv},\;\forall h\in H);
\end{equation}

\item
The left $H$-comodule structure inherits from $M$.
\end{itemize}
\end{lemma}

In this subsection, we aim to refine Theorem \ref{them: cc} to find a generalization of Lemma \ref{lem: sch94}.
Before stating the result, let us remark that the finite tensor categories  ${}_{H}\mathfrak{YD}^{K}$ and ${}^{K}\mathfrak{YD}_{H}$ mentioned in Lemma \ref{lem: tps of YD mod} are indeed tensor equivalent. This seems known, but we provide here a proof for completion.

\begin{lemma}\label{them: YDmod mce}
There exists a contravariant tensor equivalence:
\begin{equation}\label{eqn: lHYDrK equiv lKYDrH}
{}_{H}\mathfrak{YD}^{K}\approx {}^{K}\mathfrak{YD}_{H},\;\;V\mapsto V^\ast
\end{equation}
between finite tensor categories with the monoidal structure
\begin{equation}\label{J_{V, W}}
J_{V, W}: V^\ast\otimes W^\ast\rightarrow(V\otimes W)^\ast,\;\;\;\;v^\ast\otimes w^\ast\mapsto\la v^\ast\otimes w^\ast, -\ra.
\end{equation}
\end{lemma}

\begin{proof}
Let $V\in{}_{H}\mathfrak{YD}^{K}$, and we define its dual space $V^\ast$ to be canonically an object in ${}^{K}\mathfrak{YD}_{H}$.
Specifically, the right $H$-module structure on $V^\ast$ is given by
\begin{equation}\label{eqn: v* cdot h}
v^\ast\cdot h=\langle v^\ast, h\cdot(-)\rangle\;\;\;\;\;\;\;\;
(\forall h\in H,\;\forall v^\ast\in V^\ast),
\end{equation}
and the left $K$-comodule structure on $V^\ast$ is denoted by
\begin{equation}\label{eqn: lKcomod of V*}
v^\ast\mapsto\sum v^\ast_{\langle-1\rangle}\otimes v^\ast_{\langle0\rangle},\;\;\;\;\text{which satisfies}\;\;\;\;\sum v^\ast_{\langle-1\rangle}\langle v^\ast_{\langle0\rangle}, v\rangle=\sum\langle v^\ast, v_{\langle0\rangle}\rangle v_{\langle1\rangle}
\;\;\;\;(\forall v\in V).
\end{equation}
Now we verify that these structures satisfy the compatibility condition (\ref{eqn: lKcomodrHmod of YD}) in the category ${}^{K}\mathfrak{YD}_{H}$. In order to show that
\begin{equation}\label{eqn: lKcomodrHmod of V*0}
\sum(v^\ast\cdot h)_{\langle-1\rangle}\otimes (v^\ast\cdot h)_{\langle0\rangle}=\sum S^{-1}(\sigma_r(h_{(3)}))
v^\ast_{\langle-1\rangle}\sigma_r(h_{(1)})\otimes (v^\ast_{\langle0\rangle}\cdot h_{(2)})
\end{equation}
holds for any $h\in H$ and $v^\ast\in V^\ast$, we compare the images of both sides under any $\id\otimes v$ ($v\in V$) by following calculations:
\begin{eqnarray*}
&&
\sum S^{-1}(\sigma_r(h_{(3)}))
v^\ast_{\langle-1\rangle}\sigma_r(h_{(1)})\langle v^\ast_{\langle0\rangle}\cdot  h_{(2)}, v\rangle \\
&\overset{(\ref{eqn: lKcomod of V*})}=&
\sum S^{-1}(\sigma_r(h_{(3)}))(h_{(2)}\cdot v)_{\langle1\rangle}\sigma_r(h_{(1)})
\langle v^\ast, (h_{(2)}\cdot v)_{\langle0\rangle}\rangle\\
&\overset{(\ref{eqn: lHmodrKcomod YD})}=&
\sum S^{-1}(\sigma_r(h_{(5)}))\sigma_r(h_{(4)})v_{\langle1\rangle}
S^{-1}(\sigma_r(h_{(2)}))\sigma_r(h_{(1)})\langle v^\ast, h_{(3)}\cdot v_{\langle0\rangle}\rangle \\
&\overset{(\ref{eqn: lKcomod of V*})}=&
\sum(v^\ast\cdot h)_{\langle-1\rangle}\langle(v^\ast\cdot h)_{\langle0\rangle}, v\rangle.
\end{eqnarray*}
It follows that $V^\ast\in {}_H\mathfrak{YD}^K$, and hence we obtain the desired functor.

Next, $J$ is clearly a well-defined natural isomorphism, and we proceed to show that
$J_{V, W}$ is a morphism in ${}^{K}\mathfrak{YD}_{H}$ for any objects $V$ and $W$.
Let us verify that $J_{V, W}$ preserves left $K$-coactions for instance, since the right $H$-actions
are preserved by similar calculations: For any $v\in V$, $w\in W$ and $k^\ast\in K^\ast$,
\begin{eqnarray*}
&&\sum\la k^\ast, J_{V, W}(v^\ast\otimes w^\ast)_{\la-1\ra}\ra\la J_{V, W}(v^\ast\otimes w^\ast)_{\la0\ra}, v\otimes w\ra \\
&\overset{(\ref{eqn: lKcomod of V*}), (\ref{eqn: rK YD})}=&
\sum\la J_{V, W}(v^\ast\otimes w^\ast), v_{\la0\ra}\otimes w_{\la0\ra}\ra\la k^\ast, w_{\la1\ra}v_{\la1\ra}\ra \\
&\overset{(\ref{eqn: lKcomod of V*})}=&
\sum\la k^\ast_{(1)}, w^\ast_{\la-1\ra}\ra\la k^\ast_{(2)}, v^\ast_{\la-1\ra} \ra\la v^\ast_{\la0\ra}, v\ra\la w^\ast_{\la0\ra}, w\ra \\
&\overset{(\ref{J_{V, W}})}=&
\sum\la k^\ast, w^\ast_{\la-1\ra}v^\ast_{\la-1\ra}\ra\la J_{V, W}(v^\ast_{\la0\ra}\otimes w^\ast_{\la0\ra}), v\otimes w\ra,
\end{eqnarray*}
which implies that
$$\sum J_{V,W}(v^\ast\otimes w^\ast)_{\la-1\ra}\otimes J_{V,W}(v^\ast\otimes w^\ast)_{\la0\ra}
=\sum w^\ast_{\la-1\ra}v^\ast_{\la-1\ra}\otimes J_{V, W}(v^\ast_{\la0\ra}\otimes w^\ast_{\la0\ra})$$
holds for any $v^\ast\in V^\ast$ and $w^\ast\in W^\ast$.

Finally, it is evident to note that the equation
\begin{equation}
J_{U\otimes V,W}\circ(J_{U,V}\otimes\id_{W^\ast})
=J_{U,V\otimes W}\circ(\id_{U^\ast}\otimes J_{V,W})
\end{equation}
holds for any $U,V,W\in {}_{H}\mathfrak{YD}^{K}$. The proof is completed.
\end{proof}

Besides, the following lemma should be also noted.
\begin{lemma}\label{lem:coinv}
Suppose $N$ is a finite-dimensional left $K$-comodule with structure $n\mapsto\sum n_{(-1)}\otimes n_{(0)}$ which induces the right $K^\ast$-action by
$$n\cdot k^\ast=\sum\la k^\ast,n_{(-1)}\ra n_{(0)}
\;\;\;\;(\forall k^\ast\in K^\ast,\;\forall n\in N).$$
If we regard $N^\ast$ as a right $K$-comodule induced by $N$ via the duality functor ${}^K\M\approx\M^K$, then the space $N^\ast_\mathrm{coinv}$ of its coinvariants coincides with the image of the injection
\begin{equation}\label{eqn: jmath}
q^\ast: (N/(N\cdot(K^\ast)^+))^\ast\rightarrowtail N^\ast,\;\;f\mapsto f\circ q
\end{equation}
induced by the quotient map $q:N\twoheadrightarrow N/(N\cdot(K^\ast)^+)$.
\end{lemma}

\begin{proof}
First we know that the image of $q^\ast$ is
$$\mathrm{Im}(q^\ast)=\{n^\ast\in N^\ast\mid \la n^\ast, N\cdot(K^\ast)^+\ra=0\}.$$
Now let us consider $N^\ast$ again as the left $K^\ast$-module canonically with structure $\cdot$ satisfying that
$$
\la k^\ast\cdot n^\ast, n\ra=\sum\la k^\ast, n_{(-1)}\ra\la n^\ast, n_{(0)}\ra=\la n^\ast, n\cdot k^\ast\ra
$$
hold for all $k^\ast\in K^\ast$, $n^\ast\in N^\ast$ and $n\in N$. It follows that
$$
\mathrm{Im}(q^\ast)=\{n^\ast\in N^\ast\mid \la(K^\ast)^+\cdot n^\ast, N\ra=0\}
=\{n^\ast\in N^\ast\mid \forall k^\ast\in K^\ast,\;\;k^\ast\cdot n^\ast=\la k^\ast, 1\ra n^\ast\}
$$
is exactly the space of invariants of the left $K^\ast$-module $N^\ast$. Then according to \cite[Lemma 1.7.2(1)]{Mon93}, we find $\mathrm{Im}(q^\ast)= N^\ast_\mathrm{coinv}$ as a consequence.
\end{proof}

We end this paper by establishing the following tensor equivalence,
and
Schauenburg's characterization (Lemma \ref{lem: sch94})
is exactly the situation
when $K=H$ and $\sigma$ is the evaluation.

\begin{proposition}\label{prop: four equiv KYDH}
There is an equivalence of finite tensor categories
\begin{equation}\label{eqn: four equiv KYDH}
 {}_K^K\M^K_H\approx {}^K\mathfrak{YD}_H,\;\;\; M\mapsto M_\mathrm{coinv},
\end{equation}
which sends each $M\in{}_K^K\M_H^K$ to the space $M_\mathrm{coinv}$ of its coinvariants as a right $K$-comodule, with structures as follows:
\begin{itemize}
\item
The right $H$-module structure $\triangleleft$ is given by
\begin{equation}\label{eqn: rH of McoK}
m\triangleleft h=\sum S^{-1}(\sigma_r(h_{(2)}))\cdot m\cdot h_{(1)}\;\;\;\;\;\;\;\;
(\forall m\in M_\mathrm{coinv},\;\forall h\in H);
\end{equation}

\item
The left $K$-comodule structure inherits from $M$.
\end{itemize}
\end{proposition}

\begin{proof}
The desired structures on $M_\mathrm{coinv}$ are clearly well-defined.

Note that the right $K$-coaction of $M$ induces canonically the left $K$-coaction of $M^\ast$. Then according to Lemma \ref{lem:coinv}, we have a linear isomorphism
\begin{equation}\label{eqn: p*}
q^\ast:(\overline{M^\ast})^\ast=\left(M^\ast/(M^\ast\cdot(K^\ast)^+)\right)^\ast
\cong M^\dast_\mathrm{coinv}
\end{equation}
induced by the quotient map
$$q:M^\ast\rightarrow \overline{M^\ast}=M^\ast/(M^\ast\cdot(K^\ast)^+),
\;\;m^\ast\mapsto \overline{m^\ast}.$$
Here, $M^\dast_\mathrm{coinv}$ is the space of coinvariants of the right $K$-comodule $M^\dast$ which is in fact canonically isomorphic to $M$.

From now on, we make the identification $M^\dast=M$ as objects in ${}_K^K\M^K_H$, and then it follows that $M^\dast_\mathrm{coinv}=M_\mathrm{coinv}$. However, one can find that $(\overline{M^\ast})^\ast$ is exactly the image of $M$ under the composition of (\ref{eqn: lHYDrK equiv lKYDrH}) and (\ref{eqn: four equiv HYDK}). Therefore, our goal is to show that $q^\ast$ is an isomorphism in ${}^K\YD_H$, which will imply that (\ref{eqn: four equiv KYDH}) is also a tensor functor.

To this end, consider the left $H$-action (resp. right $K$-action) on $M^\ast$ induced by the right $H$-action (resp. left $K$-action) on $M\in{}_K^K\M^K_H$, namely, sent by (\ref{eqn: four* cong four inverse}). Thus we can write
\begin{equation}\label{eqn: h m* m}
\begin{array}{r}
\la h\cdot m^\ast,m\ra=\la m^\ast,m\cdot h\ra
\;\;\;\;\text{and}\;\;\;\;
\la m^\ast\cdot k, m\ra=\la m^\ast,k\cdot m\ra
\;\;\;\;\;\;\;\;\;\;\;\;\;\;\;\;\;\;\;\;\;\;\;\;\;\;\;\;\;\;\;\;\;\;\;\;
  \\
(\forall h\in H,\;\forall k\in K,\;\forall m^\ast\in M^\ast,\;\forall m\in M),
\end{array}
\end{equation}
or equivalently with in the notation of (\ref{eqn: comod P*}) in Remark \ref{rmk: four cong four*}:
\begin{equation}\label{eqn: lH M*}
h\cdot m^\ast=\sum m^{\ast(0)}\la m^{\ast(1)}, h\ra,
\;\;\;\;
m^\ast\cdot k=\sum\la m^{\ast(-1)},k\ra m^{\ast(0)}
\;\;\;\;
(\forall h\in H,\;\forall k\in K,\;\forall m^\ast\in M^\ast).
\end{equation}
Besides,
$\overline{M^\ast}\in\Rep(K^{\ast\cop}\bowtie_\sigma H)$ is also a left $H$-module whose structure is given by \cite[(4.15) and (3.14)]{Li23}
\begin{eqnarray*}
h\cdot \overline{m^\ast}
&:=& \sum \overline{m_{(0)}^\ast}\la m_{(1)}^\ast,\iota(h)\ra
\overset{(\ref{def: iota})}=
\sum \overline{m_{(0)}^\ast}
  \big\la m_{(1)}^\ast,\sigma_r(S^{-1}(h_{(2)}))\otimes h_{(1)}\big\ra  \\
&\overset{(\ref{eqn: rK*coptH*com of M})}=&
\sum \la m^{\ast(-1)},\sigma_r(S^{-1}(h_{(2)}))\ra
  \overline{m^{\ast(0)}}\la m^{\ast(1)}, h_{(1)}\ra  \\
&\overset{(\ref{eqn: lH M*})}=&
\overline{\sum h_{(1)}\cdot m^\ast\cdot\sigma_r(S^{-1}(h_{(2)}))}
\end{eqnarray*}
for any $h\in H$ and $m^\ast\in M^\ast$. As a consequence, for any $f\in (\overline{M^\ast})^\ast$, we find that
\begin{eqnarray*}
\la q^\ast(f\cdot h),m^\ast\ra
&\overset{(\ref{eqn: p*})}=&
\la f\cdot h,\,\overline{m^\ast}\ra
\overset{(\ref{eqn: v* cdot h})}=
\la f,h\cdot\overline{m^\ast}\ra
=
\Big\la f,\,\overline{\sum h_{(1)}\cdot m^\ast\cdot\sigma_r(S^{-1}(h_{(2)}))}\Big\ra  \\
&\overset{(\ref{eqn: p*})}=&
\Big\la q^\ast(f),\,\sum h_{(1)}\cdot m^\ast\cdot\sigma_r(S^{-1}(h_{(2)}))\Big\ra,
\\
&\overset{(\ref{eqn: h m* m})}=&
\Big\la \sum\sigma_r(S^{-1}(h_{(2)}))\cdot q^\ast(f)\cdot h_{(1)},\,m^\ast\Big\ra
\\
&=&
\la q^\ast(f)\triangleleft h,\,m^\ast \ra.
\end{eqnarray*}
Therefore, $q^\ast$ (\ref{eqn: p*}) is a right $H$-module map.

On the other hand, note that $(\overline{M^\ast})^\ast\in {}^K\YD_H$ and that $M_\mathrm{coinv}$ is a left $K$-submodule of $M=M^\dast\in {}_K^K\M^K_H$.
Thus using our previous notation, we should verify that
\begin{equation}\label{eqn: leftKcomodmap}
\sum f_{\la-1\ra}\otimes q^\ast(f_{\la0\ra})
=\sum q^\ast(f)^{(-1)}\otimes q^\ast(f)^{(0)}
\;\;\;\;\;\;\;\;(\forall f\in(\overline{M^\ast})^\ast )
\end{equation}
holds in $K\otimes M^\dast$, which means that $q^\ast$ preserves right $K$-coactions.
To this end,
it follows from the sentence before \cite[(4.15)]{Li23} as well as \cite[(3.12)]{Li23} that $\overline{M^\ast}$ is defined to be the quotient left module of $M^\ast$ over $K^{\ast\cop}$ (or $K^\ast$), and one can write
\begin{equation}\label{eqn: quotientleftmod}
(k^\ast\bowtie 1)\cdot \overline{m^\ast}=\overline{k^\ast\cdot m^\ast}\;\;\;\;\;\;\;\;
(\forall k^\ast\in K^\ast,\;\forall m^\ast\in M^\ast),
\end{equation}
where the left $K^\ast$-action on $M^\ast$ is given by (\ref{eqn: mod P*}).
Then we compare the images of both sides of (\ref{eqn: leftKcomodmap}) under any $k^\ast\otimes m^\ast$ $(k^\ast\in K^\ast,\;m^\ast\in M^\ast)$ in the following calculation:
\begin{eqnarray*}
\sum \la k^\ast,f_{\la-1\ra}\ra\la m^\ast,q^\ast(f_{\la0\ra})\ra
&\overset{(\ref{eqn: p*})}=&
\sum \la k^\ast,f_{\la-1\ra}\ra\la f_{\la0\ra},\overline{m^\ast}\ra
\overset{(\ref{eqn: lKcomod of V*})}=
\sum \la f,\overline{m^\ast}_{\la0\ra}\ra\la k^\ast,\overline{m^\ast}_{\la1\ra}\ra \\
&\overset{(\ref{eqn: YD iso repqd action})}=&
\sum \la f,(k^\ast\bowtie 1)\cdot\overline{m^\ast} \ra
\overset{(\ref{eqn: quotientleftmod})}=
\sum \la f,\,\overline{k^\ast\cdot m^\ast} \ra  \\
&\overset{(\ref{eqn: p*})}=&
\sum \la k^\ast\cdot m^\ast,q^\ast(f)\ra
\overset{(\ref{eqn: mod P*})}=
\sum \la m^\ast,q^\ast(f)\leftharpoonup k^\ast\ra
\\
&\overset{(\ref{eqn: hit action})}=&
\sum \la k^\ast,q^\ast(f)^{(-1)}\ra\la m^\ast,q^\ast(f)^{(0)}\ra
\end{eqnarray*}
for any $f\in (\overline{M^\ast})^\ast$, and hence Equation (\ref{eqn: leftKcomodmap}) is concluded.
\end{proof}

\begin{remark}
As the composition of quasi-inverses of (\ref{eqn: lHYDrK equiv lKYDrH}) and (\ref{eqn: four equiv HYDK})
sends each $V\in {}^K\YD_H$ to
$(V^\ast\otimes K^\ast)^\ast$, which can be isomorphic to $V\otimes K\mapsfrom V$ as objects in ${}_K^K\M^K_H$.
Therefore,
one may verify that the tensor functor (\ref{eqn: four equiv KYDH}) has quasi-inverse of form $V\otimes K\mapsfrom V$, and this is a special case of \cite[Theorem 3.1]{BDRV98} as a $\k$-linear abelian equivalence.
\end{remark}

\section*{CRediT authorship contribution statement}
\textbf{Ji-Wei He:}
Writing - review and editing, Methodology, Supervision, Validation.
\textbf{Xiaojie Kong:}
Writing - original draft, Conceptualization, Methodology, Validation.
\textbf{Kangqiao Li:}
Writing - review and editing, Conceptualization, Methodology, Validation.

\section*{Declaration of competing interest}

The authors declare that they have no known competing financial interests or personal relationships that could have appeared to influence the work reported in this paper.


\section*{Acknowledgements}
The work was partially supported by National Natural Science Foundation of China [grant numbers 12301049, 12371017].

The authors are grateful to the referee for the detailed and valuable comments, which help to greatly improve the quality of this article.
The authors would also like to thank Professor Xiaolan Yu for useful discussions at her seminars, and thank Professor Shenglin Zhu for backgrounds and suggestions.

\end{document}